\documentclass{amsart}
\usepackage{amssymb}
\usepackage{amsmath}
\usepackage{amsfonts}
\usepackage{amsxtra}
\usepackage{color}
\usepackage{xcolor}

\numberwithin{equation}{section}

\theoremstyle{plain}
\newtheorem{theorem}{Theorem}[section]
\newtheorem{proposition}[theorem]{Proposition}
\newtheorem{lemma}[theorem]{Lemma}
\newtheorem{corollary}[theorem]{Corollary}

\newtheorem*{Lindtheorem}{Lindenstrauss' pointwise ergodic theorem}

\theoremstyle{remark}

\theoremstyle{definition}

\newcommand{\reft}[1]{Theorem \ref{#1}}
\newcommand{\refp}[1]{Proposition \ref{#1}}
\newcommand{\refl}[1]{Lemma \ref{#1}}

\newcommand{\refs}[1]{Section \ref{#1}}

\newcommand{\refeq}[1]{Eq. \eqref{#1}}

\def\bet{\beta}

\def\del{\delta}

\def\lam{\lambda}
\def\tet{\vartheta}
\def\vphi{\varphi}

\def\Ome{\Omega}

\def\CA{\mathcal{A}}
\def\CB{\mathcal{B}}
\def\CC{\mathcal{C}}
\def\CD{\mathcal{D}}
\def\CE{\mathcal{E}}
\def\CF{\mathcal{F}}
\def\CG{\mathcal{G}}
\def\CH{\mathcal{H}}
\def\CK{\mathcal{K}}
\def\CL{\mathcal{L}}
\def\CM{\mathcal{M}}

\def\CO{\mathcal{O}}

\def\CR{\mathcal{R}}
\def\CS{\mathcal{S}}
\def\CT{\mathcal{T}}

\def\AC{\CA^{\circ}}
\def\GC{\CG^{\circ}}

\def\k{\Bbbk}
\def\kc{\k^{\circ}}

\def\U{\mathbf{U}}
\def\C{\mathbb{C}}
\def\F{\mathbb{F}}
\def\I{\mathbb{I}}
\def\N{\mathbb{N}}
\def\G{\Gamma}
\def\R{\mathbb{R}}

\def\fq{\mathbb{F}_{q}}

\def\fru{\mathfrak{u}}

\def\g{\gamma}
\def\1{\mathbf{1}}
\def\0{\mathbf{0}}
\def\m{\mu}

\newcommand{\map}[3]{#1 \colon #2 \to #3}
\newcommand{\seq}[2]{#1_{1}, \ldots, #1_{#2}}
\newcommand{\set}[2]{\{ #1 \colon #2 \}}

\def\x{\times}
\def\inv{^{-1}}
\def\trp{^{\mathtt{T}}}
\def\sset{\subseteq}
\def\ovl{\overline}
\def\Ch{\operatorname{Ch}}
\def\ICh{\Ch^{+}}
\def\Irr{\operatorname{Irr}}
\def\Cl{\operatorname{Cl}}
\def\SCl{\operatorname{SCl}}
\def\SCh{\operatorname{SCh}}
\def\ISCh{\SCh^{+}}

\def\Ind{\operatorname{Ind}}

\def\ev{\boldsymbol{e}}
\def\fc{\mathrm{fc}}
\def\fsc{\mathrm{fsc}}

\def\id{\operatorname{id}}

\def\SP{\operatorname{\mathbf{Sp}}}
\def\Col{\operatorname{Col}}
\def\tr{\operatorname{Tr}}
\def\Hom{\operatorname{Hom}}
\def\supp{\operatorname{supp}}

\def\nest{\operatorname{nest}}

\allowdisplaybreaks

\begin{document}


\title[]{Supercharacters of discrete algebra groups}

\author[]{Carlos A. M. Andr\'e}

\author[]{Jocelyn Lochon}

\thanks{This research was made within the activities of the Group for Linear, Algebraic and Combinatorial Structures of the Center for Functional Analysis, Linear Structures and Applications (University of Lisbon, Portugal), and was partially supported by the Portuguese Science Foundation (FCT) through the Strategic Projects UID/MAT/04721/2013 and UIDB/04721/2020. The second author was partially supported by the Lisbon Mathematics PhD program (funded by the Portuguese Science Foundation). The major part of this work is included in the second author Ph.D. thesis.}

\address[C. A. M. Andr\'e]{Centro de An\'alise Funcional, Estruturas Lineares e Aplica\c c\~oes (Grupo de Estruturas Lineares e Combinat\'orias) \\ Departamento de Matem\'atica \\ Faculdade de Ci\^encias da Universidade de Lisboa \\ Campo Grande, Edif\'\i cio C6, Piso 2 \\ 1749-016 Lisboa \\ Portugal}
\email{caandre@ciencias.ulisboa.pt}

\address[J. Lochon]{Centro de An\'alise Funcional, Estruturas Lineares e Aplica\c c\~oes (Grupo de Estruturas Lineares e Combinat\'orias) \\ Departamento de Matem\'atica \\ Faculdade de Ci\^encias da Universidade de Lisboa \\ Campo Grande, Edif\'\i cio C6, Piso 2 \\ 1749-016 Lisboa \\ Portugal}
\email{jocelyn.lochon@gmail.com}

\subjclass[2010]{22D10; 22D40}

\date{\today}

\keywords{Amenable countable discrete algebra group; supercharacter; ergodic measure}

\begin{abstract}
The concept of a supercharacter theory of a finite group was introduced by Diaconis and Isaacs in \cite{Diaconis2008a} as an alternative to the usual irreducible character theory, and exemplified with a particular construction in the case of finite algebra groups. We extend this construction to arbitrary countable discrete algebra groups, where superclasses and indecomposable supercharacters play the role of conjugacy classes and indecomposable characters, respectively. Our construction can be understood as a cruder version of Kirillov's orbit method and a generalisation of Diaconis and Isaacs construction for finite algebra groups. However, we adopt an ergodic theoretical point of view. The theory is then illustrated with the characterisation of the standard supercharacters of the group of upper unitriangular matrices over an algebraic closed field of prime characteristic.
\end{abstract}

\maketitle

\section{Introduction}

This paper mainly deals with the unitary representation theory of countable discrete algebra groups. Let $\k$ be an arbitrary field, and let $\CA$ be an associative nil $\k$-algebra; we recall that an associative $\k$-algebra $\CA$ is said to be \textit{nil} if every element of $\CA$ is nilpotent (in particular, $\CA$ does not have an identity). Let $G = 1+\CA$ be the set of formal objects of the form $1+a$ where $a \in \CA$; then, $G$ is easily seen to be a group with respect to the multiplication defined by $(1+a)(1+b) = 1+a+b+ab$ for all $a,b \in \CA$. (In fact, $G$ is a subgroup of the group of units of the $\k$-algebra $\CA_{1} = \k\cdot 1 + \CA$.) Following \cite{Isaacs1995a}, a group $G$ constructed in this way will be referred to as an \textit{algebra group over $\k$}; by the way of example, if $\CA = \fru_{n}(\k)$ is the $\k$-algebra consisting of all strictly uppertriangular $n \x n$ matrices over $\k$, then the corresponding algebra group $G = 1+\CA$ is (isomorphic to) the upper unitriangular group $U_{n}(\k)$. Henceforth, we view $G$ as a subgroup of the group of units of the $\k$-algebra $\k\cdot 1 + \CA$; however, we observe that $\CA$ is not assumed to have finite dimension over $\k$.

In general, countable discrete algebra groups are not \textit{tame} (or of type I), meaning that the decomposition of an arbitrary unitary representation (on a separable Hilbert space) as a direct integral of irreducible representations might not be unique; indeed, this is not the case if and only if the group in question has an abelian subgroup of finite index, as shown by E. Thoma in \cite[Satz 6]{Thoma1964a} (see also \cite{Tonti2019a}). For this reason, we will approach the representation theory of a discrete algebra group through its character theory. 

More generally, let $G$ be an arbitrary topological group. A complex-valued continuous function $\map{\vphi}{G}{\C}$ is said to be \textit{positive definite} if the following two conditions are satisfied:
\begin{enumerate}
\item $\vphi(g^{-1}) = \overline{\vphi(g)}$ for all $g \in G$.
\item For every $\seq{g}{m} \in G$, the Hermitian matrix $\big[\vphi(g_{i}g_{j}^{-1})\big]_{1 \leq i,j \leq m}$ is nonnegative, that is, $\sum_{1 \leq i,j \leq n} \vphi(g_ig_j^{-1})z_i\overline{z_j} \geq 0$ for all $\seq{z}{n} \in \C.$
\end{enumerate}
A continuous function $\map{\vphi}{G}{\C}$ is said to be \textit{central} (or a \textit{class function}) if it is constant on the conjugacy classes of $G$, that is, if $\vphi(ghg^{-1}) = \vphi(h)$ for all $g,h \in G$, and it is said to be \textit{normalised} if $\vphi(1) = 1$. We denote by $\Ch(G)$ the set consisting of all normalised, central, positive definite continuous functions $\map{\vphi}{G}{\C}$.  If $\vphi,\psi \in \Ch(G)$, then for every real number $0 \leq t \leq 1$ the function $(1-t)\vphi+t\psi$ is also an element of $\Ch(G)$, which means that $\Ch(G)$ is a convex set. An element of a convex set is said to be \textit{extreme} if it is not contained in the interior of any interval which is entirely contained in the set; we denote by $\ICh(G)$ the subset of $\Ch(G)$ consisting of extreme elements of $\Ch(G)$. We refer to an element of $\Ch(G)$ as a \textit{character} of $G$ and to an element of $\ICh(G)$ as an \textit{indecomposable character} of $G$. If $G$ is a compact group (in particular, a finite group), then $\ICh(G)$ consists of all \textit{normalised irreducible characters} of $G$; the character of an irreducible (complex) representation $\pi$ of the group $G$ is the function $\chi \colon G \to \mathbb{C}$ given by the trace $\chi(g) = \mathrm{tr}(\pi(g))$ for all $g \in G$, while the normalised irreducible character is given by $\widehat{\chi} = \chi(1)\inv \chi$. This definition makes sense because every irreducible representation of a compact group is finite-dimensional, but it does not make sense in general for non-compact groups. 

In the case where $G$ is a discrete countable group, then $\Ch(G)$ is a Choquet simplex, with respect to the topology of pointwise convergence, such that $\ICh(G)$ is a $G_{\delta}$-set (as proved by Thoma in \cite{Thoma1964a,Thoma1968a}). Consequently, every character of $G$ is uniquely representable as a convex mixture of indecomposable characters (see, for example, \cite{Goodearl1986a} or \cite{Phelps2001a}): for every $\vphi \in \Ch(G)$, there is a unique probability measure $\mu^{\vphi}$ on $\ICh(G)$ such that $$\vphi(g)=\int_{\ICh(G)}\chi(g)\; d\mu^{\vphi}, \qquad g \in G;$$ we refer to $\mu^{\vphi}$ as the \textit{Choquet (or spectral) measure associated with} $\vphi$. (Notice that the uniqueness of this representation of a character with respect to a Choquet measure mirrors the classical decomposition of characters for finite groups as a sum of irreducible characters; moreover, it allows to view $\ICh(G)$ as a \textit{dual space} for $G$, in contrast to the usual dual space consisting of all classes of irreducible representations which for non-tame groups, not only does not provide unique decomposition, but also lacks good topological features; we refer to \cite{Dixmier1977a} fore more details.)

The main tool to study characters of discrete groups is the Gelfand-Naimark-Segal construction which associates a unitary cyclic representation with every character in such a way that indecomposable characters are paired with factor representations of finite von Neumann type. Let $G$ be an arbitrary discrete group, and let $\map{\vphi}{G}{\C}$ be a character of $G$. Let $\C[G]$ denote the complex group algebra of $G$ which is naturally equipped with the involution $\map{\ast}{\C[G]}{\C[G]}$ defined by the rule $\big(\sum_{g \in G} a_{g}g\big)^{\ast} = \sum_{g \in G} \ovl{a}_{g}g\inv$ where $a_{g} \in \C$ for $g \in G;$ in particular, $\C[G]$ ia a $\ast$-algebra. By abuse of notation, we denote the linear extension of $\vphi$ to $\C[G]$ also by $\vphi$; since $\vphi$ is positive-definite, the formula $\langle a, b \rangle = \vphi(b^{\ast}a)$, for $a,b \in \C[G]$. defines a positive semi-definite sesquilinear form on $\C[G]$. Thus, via separation and completion, we obtain a Hilbert space $\CH_{\vphi}$ equipped with a natural map $G \to \CH_{\vphi}$ which sends an element $g \in G$ to the characteristic function $\del_{g}$ of $\{g\}$. It is a standard fact that the left action of $G$ on $\C[G]$ induces a unitary representation $\map{\pi_{\vphi}}{G}{\U(\CH_{\vphi})}$ such that $\vphi(g) = \langle \pi(g)\del_{1}, \del_{1} \rangle$ for all $g \in G$; following the terminology of \cite{Bekka2020c} , we will refer to $(\pi_{\vphi},\CH_{\vphi})$ as the \textit{GNS representation of $G$ associated with $\vphi$}. Let $\CL(G,\vphi)$ denote the von Neumann algebra generated by $\pi_{\vphi}(G)$; hence, $\CL(G,\vphi) = \pi_{\vphi}(G)''$ is the bicommutant of $\pi_{\vphi}(G)$. Since $\vphi$ is a character, the mapping $x \mapsto \vphi(x) = \langle x\del_{1}, \del_{1} \rangle$ defines a faithful normal unital trace on $\CL(G,\vphi)$, and thus $\CL(G,\vphi)$ is a finite von Neumann algebra. Finally, we mention that the central decomposition of $\CL(G,\vphi)$ as a direct integral of factor representations corresponds uniquely to the Choquet decomposition of $\vphi \in \Ch(G)$ as an integral of indecomposable characters; in particular, $\vphi \in \ICh(G)$ if and only if $\CL(G,\vphi)$ is a factor. For more details on von Newmann algebras, we refer to Dixmier's book \cite{Dixmier1981a}; see also \cite{Dixmier1977a}.

Indecomposable characters have been explicitly classified for concrete groups (see for example \cite{Dudko2011a,Hirai2005b,Corwin1994a,Vershik1981a,Vershik1982a}). However, the set $\ICh(G)$ may be too large or too complicated to describe (even in the case of a finite group), and thus it might be of interest to consider smaller and more manageable families of characters which are still rich enough to provide some relevant information about their representations. This question was first addressed in the context of finite groups by P. Diaconis and I.M. Isaacs in the foundational paper \cite{Diaconis2008a} where the notion of \textit{supercharacter theory} was formalised (motivated by previous work of C. Andr\'e \cite{Andre1995a,Andre1995b,Andre1998a,Andre2001a,Andre2002a} and of N. Yan \cite{Yan2001a} on the character theory of the unitriangular groups over finite fields). By a \textit{supercharacter theory} of a finite group $G$ we mean a pair $(\CK,\CE)$ where $\CK$ is a partition of $G$, and $\CE$ is an orthogonal set of characters of $G$ satisfying:
\begin{enumerate}
\item $|\CK| = |\CE|$,
\item every character $\xi \in \CE$ takes a constant value on each member $K \in \CK$, and
\item each irreducible character of $G$ is a constituent of one of the characters $\xi \in \CE$.
\end{enumerate}
We refer to the members $K \in \CK$ as \textit{superclasses} and to the characters $\xi \in \CE$ as \textit{indecomposable supercharacters}\footnote{It is worth to mention that, here and throughout the paper, we have chosen to use a terminology which differs from the most common one; indeed, in \cite{Diaconis2008a} the term ``supercharacter'' is reserved to the elements of $\CE$ (that is, to the indecomposable supercharacters). We prefer to use the term ``supercharacter'' with the meaning of a ``character which is constant on superclasses''; this is in fact consistent with the fact that an arbitrary character is not necessarily an irreducible (or, indecomposable) character.} of $G$. Notice that superclasses of $G$ are always unions of conjugacy classes; moreover, $\{1\}$ is a superclass and the trivial character $1_{G}$ is always a supercharacter of $G$. The superclasses and the indecomposable supercharacters of a particular supercharacter theory exhibit much of the same duality as the conjugacy classes and the irreducible characters of the group do (and thus a supercharacter theory can be interpreted as an \textit{approximation} of the classical character theory); indeed, the usual character theory of $G$ is a trivial example of a supercharacter theory where $\CK$ is the set $\Cl(G)$ consisting of all conjugacy classes of $G$ and $\CE$ is the set $\Irr(G)$ consisting of all irreducible characters of $G$.

The \textit{standard supercharacter theory} of an arbitrary finite algebra group is described in \cite{Diaconis2008a}, and the main goal of this paper is to extend this construction to the case of an arbitrary countable discrete algebra group. Prototype examples of countable discrete algebra groups are the unitriangular groups defined over an arbitrary countable discrete field $\k$. On the one hand, for every $n \in \N$, let $U_{n}(\k)$ denote the \textit{unitriangular group} over $\k$ consisting of all $n \times n$ upper-triangular matrices over $\k$ with all entries in the main diagonal equal to $1$; notice that $U_{n}(\k) = 1+\fru_{n}(\k)$ where $\fru_{n}(\k)$ is the nil $\k$-algebra consisting of all $n\x n$ upper-triangular matrices with zeroes on the main diagonal. On the other hand, let $U_{\infty}(\k)$ denote the \textit{locally finite dimensional unitriangular group} over $\k$ consisting of all infinite upper-triangular square matrices over $\k$ with all diagonal entries equal to $1$ and such that every element has only a finite number of non-zero entries above the main diagonal. For every $n \in \N$, the unitriangular group $U_{n}(\k)$ may be naturally identified as the subgroup of $U_{n+1}(\k)$ consisting of all matrices $x \in U_{n+1}(\k)$ which satisfy $x_{i,n+1} = 0$ for all $1 \leq i \leq n$, and thus $U_{\infty}(\k)$ may be realised as the direct limit $$U_{\infty}(\k) = \varinjlim_{n \in \N} U_{n}(\k) = \bigcup_{n \in \N} U_{n}(\k).$$ We also note that $U_{\infty}(\k) = 1+\fru_{\infty}(\k)$ where $\fru_{\infty}(\k) = \bigcup_{n\in \N} \fru_{n}(\k)$ is the locally finite dimensional nil $\k$-algebra consisting of all infinite upper-triangular square matrices over $\k$ with zeroes on the main diagonal; notice also that $\fru_{\infty}(\k)$ is naturally isomorphic to the direct limit $ \varinjlim_{n \in \N} \fru_{n}(\k)$..

The superclasses of an arbitrary algebra group $G = 1+\CA$ are easy to describe; since there is no danger of ambiguity, we will abbreviate the terminology and refer to a standard superclass simply as a superclass of $G$. Indeed, the group $\G = G\x G$ acts naturally on the left of $\CA$ by the rule $$(g,h) \cdot a = gah\inv,\qquad g,h \in G,\ a\in\CA.$$ Then, the $\k$-algebra $\CA$ is partitioned into $\G$-orbits $\G\cdot a = GaG$ for $a \in \CA$, and this determines a partition of the algebra group $G = 1+\CA$ into subsets $1 + \G\cdot a = 1+GaG$ for $a \in \CA$; these are precisely what we define as the \textit{superclasses} of $G$. We use the notation $\SCl(G)$ to the denote the set consisting of all superclasses of $G$; notice that $\{1\} \in \SCl(G)$ and that every $K \in \SCl(G)$ is a union of conjugacy classes.

We next define the set of supercharacters of $G$; as in the case of superclasses, by a supercharacter of $G$ we will always understand a standard supercharacter. A (continuous) function $\vphi \colon G \to \C$ is said to be \textit{supercentral} (or a \textit{superclass function}) if it is constant on the superclasses, that is, if $\vphi(1+gah) = \vphi(1+a)$ for all $g,h \in G$ and all $a \in \CA$. We denote by $\SCh(G)$ the set consisting of all normalised, supercentral, positive definite continuous functions defined on $G$; it is clear that $\SCh(G) \subseteq \Ch(G)$. As in the case of characters, $\SCh(G)$ is a convex set (in fact, a Choquet simplex); we denote by $\ISCh(G)$ the subset of $\SCh(G)$ consisting of all extreme elements of $\SCh(G)$. We refer to an element of $\SCh(G)$ as a \textit{supercharacter} of $G$, and to an element of $\ISCh(G)$ as an \textit{indecomposable supercharacter} of $G$.

In the case where $G=1+\CA$ is a finite algebra group, the indecomposable supercharacters are parametrised by the orbits of the contragradient action of $\G = G\x G$ on the \textit{Pontryagin dual} $\AC$ of the additive group $\CA^{+}$ of $\CA$ (hence, $\AC$ consists of all unitary characters $\map{\vartheta}{\CA^{+}}{\C^{\x}}$). For each of the natural actions of $G$ on $\CA$, there is a corresponding contragradient action of $G$ on $\AC$: given $\vartheta \in \AC$ and $g\in G$, we define $g\vartheta, \vartheta g \in \AC$ by the formulas $(g\vartheta)(a) = \vartheta(g\inv a)$ and $(\vartheta g)(a) = \vartheta(ag\inv)$ for all $a \in \CA$; thus, for every $\tet \in \AC$, we have a left $G$-orbit $G\vartheta$, a right $G$-orbit $\vartheta G$, and also a (left) $\G$-orbit $\G \cdot \vartheta = G\vartheta G$ (notice that the left and right $G$-actions on $\AC$ commute). For every $\G$-orbit $\CO \sset \AC$, we define the function $\map{\xi^{\CO}}{G}{\C^{\x}}$ by the rule
\begin{equation} \label{FSCh}
\xi^{\CO}(g) = \frac{1}{|\CO|} \sum_{\tet \in \CO} \tet(g-1), \qquad g \in G.
\end{equation}
By \cite[Theorem~5.6]{Diaconis2008a}), we know that $\ISCh(G) = \set{\xi^{\CO}}{\CO \in \Ome}$ where $\Ome$ denotes the set consisting of all $\G$-orbits on $\AC$.

The formula above may be interpreted in measure theoretical terms: the value of an indecomposable supercharacter $\xi^{\CO}(g)$ at an element $g \in G$ is given by the integral over $\CO$ evaluated at $g-1 \in \CA$ with respect to the \textit{unique} ergodic $\G$-invariant measure on $\AC$ supported on $\CO$; for an arbitrary countable discrete algebra group a similar phenomenon occurs. More precisely, in \reft{ErgodicCorrespondence} we establish a one-to-one correspondence between $\ISCh(G)$ and the set consisting of all $\G$-invariant measures on $\AC$ which are ergodic with respect to the (contragradient) action of $\G$ on $\AC$ (and this provides an alternative proof of \cite[Theorem~5.6]{Diaconis2008a}); as a consequence, we conclude that $\SCh(G)$ is indeed a Choquet simplex. While it is true that every ergodic $\G$-invariant measure must be supported on the closure of some $\G$-orbit (\refp{EOrbital}), and that, under the assumption that $G$ is amenable, the closure of an arbitrary $\G$-orbit must be the support of some ergodic $\G$-invariant measure (\refp{AmenableOrbital}), in general we can not guarantee that every indecomposable supercharacter is in one-to-one correspondence with the closure of a $\G$-orbit (this is because the closure of a $\G$-orbit may support distinct ergodic $\G$-invariant measures).

Another consequence is discussed in \refs{RegSCh} where we consider the regular character of $G$, and we obtain (\reft{FiniteSCl}) a necessary and sufficient condition for the regular character to be an indecomposable supercharacter of $G$; furthermore, the ergodic correspondence allows to establish (\refp{SuperPlancherel}) a supercharacter analogue for Thoma's Plancherel formula for countable discrete groups (see \cite[Satz~1]{Thoma1967a}). We should mention that these properties are the natural supercharacter generalisation of well-known properties of the regular character.  

Finally, in \refs{unitr} we consider approximately finite algebra groups $G$ (that is, direct limits of finite algebra groups), and use Lindenstrauss' pointwise ergodic theorem (\cite[Theorem~1.3]{Lindenstrauss2001a}) to approximate supercharacters of $G$ by supercharacters of finite algebra groups (\reft{FiniteApprox}); we illustrate this method by describing the supercharacters of the infinite unitriangular group $U_{n}(\k)$, where $\k$ is the algebraic closure of a finite field of prime characteristic.

\section{Supercharacters of algebra groups}

Throughout this section, we let $G=1+\CA$ be an arbitrary discrete countable algebra group associated with a nil algebra $\CA$ over a (countable discrete) field $\k$; furthermore, we consider the group $\G=G \times G$ acting naturally on the left of $\CA$ and on the Pontryagin dual $\AC$ of the additive group $\CA^{+}$ of $\CA$ (via the contragradient action). Our main goal is to parametrise the supercharacters of $G$ (as defined in the introduction) in terms of $\G$-invariant probability measures on $\AC$, in such a way that the indecomposable supercharacters of $G$ correspond to those measures which are ergodic with respect to the $\G$-action on $\AC$.

We equip $\AC$ with the topology induced by convergence on compact sets, which is nothing else than the topology of pointwise-convergence (because $\CA^{+}$ is discrete); we note that, since $\CA^{+}$ is abelian, $\AC$ is in fact the set consisting of all indecomposable characters of $\CA$. Furthermore, $\AC$ has a structure of an abelian topological compact group (see for example \cite[Proposition~4.35]{Folland1995a}) which we write additively and where the sum $\tet+\tet' \in \AC$ of two characters $\tet,\tet' \in \AC$ is determined by the pointwise product of functions, that is, $(\tet+\tet')(a) = \tet(a)\tet'(a)$ for all $a \in \CA$; the zero element of $\AC$ is the trivial character $\1_{\CA}$ which is constantly equal to $1$ (accordingly, we sometimes write $\0 = \1_{\CA}$).

We consider $\AC$ equipped with its Borel $\sigma$-algebra of measurable sets, and denote by $\CM(\AC)$ the vector space consisting of all finite complex regular Borel measures on $\AC$; we equip $\CM(\AC)$ with the topology of weak*-convergence. A measure $\mu \in \CM(\AC)$ is said to be \textit{$\G$-invariant} if $\m(\g \cdot B) = \m(B)$ for all $\g \in \G$ and all Borel subset $B$ of $\AC$; we denote by $\CM_{\G}(\AC)$ the vector space consisting of all $\G$-invariant measures on $\AC$. On the other hand, let $\CM^{+}(\AC)$ denote the subset of $\CM(\AC)$ consisting of all Borel probability measures, and let $\CM_{\G}^{+}(\AC)$ denote the subset of $\CM^{+}(\AC)$ consisting of all $\G$-invariant probability measures. Notice that $\CM_{\G}^{+}(\AC)$ (and hence $\CM_{\G}(\AC)$) is non-empty because the Dirac measure $\delta_{\1_{\CA}}$ supported on the trivial character $\1_{\CA} \in \AC$ is clearly a $\G$-invariant probability measure. Furthermore, $\CM_{\G}^{+}(\AC)$ is a Choquet simplex whose extreme elements are the ergodic measures in $\CM_{\G}^{+}(\AC)$ (see \cite[Section~12]{Phelps2001a} and \cite[Theorem~3.1]{Varadarajan1963a}); we recall that a $\G$-invariant probability measure $\m \in \CM_{\G}^{+}(\AC)$ is said to be \textit{ergodic} if, for every $\G$-invariant Borel subset $B$ of $\AC$, either $\m(B)=0$ or $\m(B)=1$ (equivalently, a measure $\m \in \CM_{\G}^{+}(\AC)$ is ergodic if and only if every $\G$-invariant function $f \in L^{2}(\AC,\m)$ is constant $\m$-almost everywhere; see, for example, \cite[Theorem~1.6]{Walters1982a}).

The main purpose of this section is to establish the existence of an affine homeomorphism between $\SCh(G)$ and $\CM^{+}_{\G}(\AC)$ (see \reft{ErgodicCorrespondence} below); in particular, we deduce that the indecomposable supercharacters of $G$ are in one-to-one correspondence with the ergodic measures in $\CM_{\G}^{+}(\AC)$. In order to prove it, we first assume a more general situation where $N$ is an arbitrary countable discrete group (not necessarily abelian), and $\G$ is a group which acts on the left of $N$ via continuous automorphisms; we write $\g \cdot n = \g(n)$ for all $\g \in \G$ and all $n \in N$. We define $\Ch_{\G}(N)$ to be the set of all characters of $N$ which are constant on every $\G$-orbit on $N$; notice that $\Ch_{\G}(N)$ is precisely the set of $\G$-invariant characters in $\Ch(N)$ with respect to the usual contragradient $\G$-action on $\Ch(N)$. As before, $\Ch_{\G}(N)$ is clearly a convex set, and thus we may consider the subset $\ICh_{\G}(N)$ of $\Ch_{\G}(N)$ consisting of all extreme elements of $\Ch_{\G}(N)$. Following the previous terminology, we refer to an element of $\Ch_{\G}(N)$ as a \textit{$\G$-character} of $N$, and to an element of $\ICh_{\G}(N)$ as an \textit{indecomposable $\G$-character} of $N$. Let $C^{*}(N)$ denote the group $C^*$-algebra of $N$ (see, for example, \cite[Section~13.9]{Dixmier1977a}); it follows from \cite[Proposition~6.8.7]{Dixmier1977a} that $\Ch(N)$ is a (compact convex) subset of the unit ball $\CB^*$ of the topological dual of $C^{*}(N)$, and hence \cite[Proposition 3.101]{Fabian2011a} implies that the Choquet simplex $\Ch(N)$ is metrisable (notice that in our situation the topology of weak*-convergence coincides with the topology of pointwise-convergence). If $\ICh(N)$ denotes the set of extreme elements of $\Ch(N)$ (that is, the indecomposable characters of $N$), then we know that every $\vphi \in \Ch(N)$ is the barycenter of a unique Borel probability measure $\mu^{\vphi}$ supported on $\ICh(N)$ (see \cite{Phelps2001a} for an exhaustive course on the theory of Choquet simplices; see also \cite{Alfsen1971a}). Therefore, if $C(\Ch(N))$ denotes the vector space consisting of all complex-valued continuous functions defined on $\Ch(N)$, then $$f(\vphi) = \int_{\ICh(N)} f(\tet)\; d\mu^{\varphi}(\tet), \qquad f \in C(\Ch(N));$$ in particular, if we choose the evaluation map at $n \in N$, then we easily deduce that
\begin{equation}\label{Choquet1}
\vphi(n) = \int_{\ICh(N)} \tet(n)\; d\mu^{\varphi}(\tet),\qquad n \in N
\end{equation}
(hence, $\mu^{\vphi}$ is the Choquet measure on $\ICh(N)$ associated with $\vphi$). The mapping $\vphi \mapsto \mu^{\vphi}$ defines a bijection $\Ch(N) \to \CM^{+}(\ICh(N))$ where $\CM^{+}(\ICh(N))$ denotes the set consisting of all Borel probability measures supported on $\ICh(N)$; for every $\mu \in \CM^{+}(\ICh(N))$, we denote by $\vphi^{\mu}$ the character $\vphi \in \Ch(N)$ such that $\mu^{\vphi} = \mu$, so that
\begin{equation}\label{Choquet1'}
\vphi^{\mu}(n) = \int_{\ICh(N)} \tet(n)\; d\mu(\tet),\qquad n \in N.
\end{equation}
By a well-known theorem of Bauer (see for example \cite[Proposition~11.1]{Phelps2001a}; see also \cite[Theorem~II.4.1]{Alfsen1971a}), the inverse bijection $\CM^{+}(\ICh(N)) \to \Ch(N)$ (given by the mapping $\mu \mapsto \vphi^{\mu}$) is  affine and continuous, and it is a homeomorphism if and only if $\ICh(N)$ is a closed subset of $\Ch(N)$, in which case $\Ch(N)$ is a Bauer simplex. (We mention that every metrisable Choquet simplex is affinely homeomorphic to the intersection of a decreasing sequence of metrisable Bauer simplices; see \cite[Theorem~9]{Edwards1975a}.) In particular, this holds in the case where $N$ is abelian where $\ICh(N)$ equals the Pontryagin dual $N^{\circ}$ of $N$, and hence is a compact subset of $\Ch(N)$ (recall that $N$ is discrete). The following result is essentially a consequence of \cite[Corollary~II.4.2]{Alfsen1971a}; we note that, since $\ICh_{\G}(N)$ is the subset consisting of $\G$-fixed elements of $\ICh(N)$, it is closed (and hence compact) whenever $\ICh(N)$ is a closed subset of $\Ch(N)$.

\begin{proposition} \label{SCh0}
Let $N$ be a countable discrete group such that $\ICh(N)$ is a closed subset of $\Ch(N)$, and let $\G$ be a group consisting of continuous automorphisms of $N$. Then, the mapping $\vphi \mapsto \mu^{\vphi}$ defines an affine homeomorphism between $\Ch_{\G}(N)$ and $\CM^{+}_{\G}(\ICh(N))$ with inverse given by the mapping $\mu \mapsto \vphi^{\mu}$. In particular, for every $\vphi \in \ICh_{\G}(N)$, the Choquet measure $\mu^{\vphi} \in \CM^{+}_{\G}(\ICh(N))$ is ergodic; conversely, if $\mu \in \CM^{+}_{\G}(\ICh(N))$ is ergodic, then the $\G$-character $\vphi^{\mu} \in \Ch_{\G}(N)$ is indecomposable.
\end{proposition}

\begin{proof}
Let $\vphi \in \Ch(N)$ and $\g \in \G$ be arbitrary. Since the Choquet measures $\mu^{\vphi}$ and $\mu^{\g\cdot\vphi}$ are uniquely determined by the characters $\vphi$ and $\g\cdot \vphi$, \refeq{Choquet1} implies that $\g\cdot\vphi = \vphi$ if and only if $\mu^{\vphi} = \mu^{\g\cdot\vphi}$. On the other hand, for every $n \in N$, we evaluate
\begin{align*}
(\g\cdot\vphi)(n) &= \vphi(\g\inv \cdot n) = \int_{\ICh(N)} \tet(\g\inv\cdot n)\; d\mu^{\vphi} \\ &= \int_{\ICh(N)} (\g\cdot \tet)(n)\; d\mu^{\vphi} = \int_{\ICh(N)} \tet(n)\; d(\g\inv\cdot\mu^{\vphi}),
\end{align*}
and thus $\mu^{\g\cdot\vphi} = \g\inv\cdot\mu^{\vphi}$. Therefore, $\g\cdot\vphi = \vphi$ if and only if $\mu^{\vphi} = \g\inv \cdot\mu^{\vphi}$, and this clearly implies that $\vphi \in \Ch_{\G}(N)$ if and only if $\mu^{\vphi} \in \CM^{+}_{\G}(\ICh(N))$.

Conversely, let $\m \in \CM^{+}_{\G}(\ICh(N))$ be arbitrary, and consider the character $\vphi^{\mu} \in \Ch(N)$ defined as in \refeq{Choquet1'}. Since $\m$ is $\G$-invariant, we easily deduce that $(\g\cdot\vphi^{\mu})(n) = \vphi^{\mu}(n)$ for all $\g \in \G$ and all $n \in N$, and hence $\vphi^{\mu} \in \Ch_{\G}(N)$. This completes the proof of the first assertion of the proposition.

Finally, we know that a measure in $\CM_{\G}^{+}(\ICh(N))$ is extreme if and only if it is ergodic (see \cite[Section~12]{Phelps2001a}), and thus $\mu^{\vphi} \in \CM_{\G}^{+}(\ICh(N))$ is ergodic for all $\vphi \in \ICh_{\G}(N)$.
\end{proof}

We now return to the previous situation where $G=1+\CA$ be an arbitrary discrete countable algebra group associated with a nil algebra $\CA$ over a (countable discrete) field $\k$; furthermore, we let $\G=G \times G$, and consider the natural action of $\G$ on $\AC$. For every $g \in G$, let $\map{\ev_{g}}{\AC}{\C}$ denote the evaluation map at $g-1 \in \CA$, that is, $\ev_{g}(\tet) = \tet(g-1)$ for all $\tet \in \AC$. As a consequence of Pontryagin duality theorem (see \cite[Theorem~4.31]{Folland1995a}), we easily conclude that $\set{\ev_{g}}{g \in G}$ is a linearly independent set of $C(\AC)$ (the vector space consisting of all complex continuous functions on $\AC$), which is closed under pointwise multiplication (indeed, $\ev_{1+a}\ev_{1+b}=\ev_{1+a+b}$, for all $a,b \in \CA$); moreover, in virtue of the Stone-Weierstrass theorem (see \cite[Theorem~11.3.1]{Dixmier1977a}), the $\C$-linear span $\ev(G)$ of $\set{\ev_{g}}{g \in G}$ is a dense $C^{\ast}$-subalgebra of $C(\AC)$.

The following elementary result will be useful.

\begin{lemma}\label{LemmaProduct}
For every $\mu \in \CM_{\G}^{+}(\AC)$ and every $g,h \in G$, we have $$\int_{\AC} \ev_{g} \ovl{\ev_{h}}\; d\mu = \int_{\AC}\ev_{h^{-1}g}\; d\mu.$$
\end{lemma}

\begin{proof}
Let $g,h \in G$ be arbitrary, and let $a,b \in \CA$ be such that $g=1+a$ and $h=1+b$. A straightforward calculation shows that $\ev_{h}(h\tet) = \ovl{\ev_{h^{-1}}(\tet)}$ for all $\tet \in \AC$ (recall that $h\tet = (h,1)\cdot \tet$ for all $\tet \in \AC$), and thus $$\int_{\AC} \ev_{g}(\tet) \ovl{\ev_{h}(\tet)}\; d\mu = \int_{\AC} \ev_{g}(\tet) \ev_{h^{-1}}(h^{-1}\tet)\; d\mu = \int_{\AC} \ev_{g}(h\tet)\ev_{h^{-1}}(\tet)\; d\mu$$ (because $\mu$ is $\G$-invariant). Since $\ev_{g}(h\tet) \ev_{h^{-1}}(\tet)=\ev_{h^{-1}g}(\tet)$ for all $\tet \in \AC$, we conclude that $$\int_{\AC} \ev_{g}(h\tet) \ev_{h^{-1}}(\tet)\; d\mu = \int_{\AC} \ev_{h^{-1}g}(\tet)\;d\mu,$$ and this completes the proof.
\end{proof}

Now, let $\xi \in \SCh(G)$ be an arbitrary supercharacter of $G$. Then, the mapping $\ev_{g} \mapsto \xi(g)$ extends by linearity to a $\C$-linear map $\map{\Phi_{0}}{\ev(G)}{\C}$; furthermore, since $\ev(G)$ is a dense $C^{\ast}$-subalgebra of $C(\AC)$, $\Phi_{0}$ extends by continuity to a continuous linear function $\map{\Phi}{C(\AC)}{\C}$. Therefore, by the Riesz-Markov-Kakutani representation theorem (\cite[Theorem~6.19]{Rudin1987a}), there is a unique measure $\mu^{\xi} \in \CM(\AC)$ such that $$\xi(g) = \int_{\AC} \ev_{g}(\tet)\; d\mu^{\xi}, \qquad  g \in G;$$ moreover, the fact that $\xi \in \SCh(G)$ ensures that $\mu^{\xi} \in \CM_{\G}(\AC)$.

\begin{proposition}\label{SCh}
If $\xi \in \SCh(G)$ is an arbitrary supercharacter of $G$, then $\mu^{\xi} \in \CM_{\G}(\AC)$ is a probability measure (hence, $\mu^{\xi} \in \CM_{\G}^{+}(\CA)$).
\end{proposition}

\begin{proof}
Let $B$ be an arbitrary Borel subset of $\AC$, and let $\I_{B} \in L^{1}(\AC,\mu^{\xi})$ be the corresponding indicator function. Since $\ev(G)$ is a dense subalgebra of $C(\AC)$, its image in $L^{1}(\AC,\mu^{\xi})$ is also dense, and hence there are families $$\set{\alpha_{i,n} \in \C}{1 \leq i \leq n,\ n \in \N}\quad \text{and}\quad \set{g_{i,n} \in G}{1 \leq i \leq n,\ n \in \N}$$ such that $$\I_{B}(\tet) = \lim_{n \to \infty} \sum_{1 \leq i \leq n} \alpha_{i,n} \ev_{g_{i,n}}(\tet),\qquad \text{for $\mu$-almost all $\tet \in \AC$.}$$ Since $\I_{B}(\tet) = \I_{B}(\tet)\,\ovl{\I_{B}(\tet)}$ for all $\tet \in \AC$, we deduce that $$\mu^{\xi}(B) = \lim_{n \to \infty} \sum_{1 \leq i,j \leq n} \alpha_{i,n}\, \overline{\alpha_{j,n}} \int_{\AC} \ev_{g_{i,n}}(\tet)\,\ovl{\ev_{g_{j,n}}(\tet)}\; d\mu^{\xi},$$ and thus \refl{LemmaProduct} implies that
\begin{align*}
\mu^{\xi}(B) &= \lim_{n \to \infty} \sum_{1 \leq i,j\leq n} \alpha_{i,n}\,\ovl{\alpha_{j,n}} \int_{\AC} \ev_{g_{j,n}\inv g_{i,n}}(\tet)\; d\mu^{\xi} \\ &= \lim_{n \to \infty} \sum_{1 \leq i,j \leq n} \alpha_{i,n}\,\ovl{\alpha_{j,n}}\, \xi(g_{j,n}\inv g_{i,n}).
\end{align*}
Since $\xi \in \SCh(G)$ is a positive-definite function, we know that $$\sum_{1 \leq i,j \leq n} \alpha_{i,n}\,\ovl{\alpha_{j,n}}\, \xi(g_{j,n}\inv g_{i,n}) \geq 0, \qquad n \in \N,$$ and thus $\mu^{\xi}(B) \geq 0$. Finally, since $\ev_{1}(\tet) = \tet(0) = 1$, we see that $$\mu^{\xi}(\AC) = \int_{\AC} \ev_{1}(\tet)\; d\mu^{\xi}  = \xi(1) = 1,$$ and so $\mu$ is a probability measure, as required.
\end{proof}

In virtue of \refp{SCh0}, we conclude that, for every supercharacter $\xi \in \SCh(G)$, there is a unique $\G$-character $\zeta \in \Ch_{\G}(\CA)$ such that $$\zeta(a) = \int_{\AC} \tet(a)\, d\mu^{\xi}, \qquad a \in \CA,$$ and thus $\mu^{\xi} = \mu^{\zeta}$ is the Choquet measure which is associated with $\zeta$; in particular, we see that $\xi(g) = \zeta(g-1)$ for all $g \in G$. Conversely, for every $\mu \in \CM^{+}_{\G}(\AC)$, we consider the $\G$-character $\zeta^{\mu} \in \Ch_{\G}(\CA)$ and define the function $\map{\xi^{\mu}}{G}{\C}$ by $\xi^{\mu}(g) = \zeta^{\mu}(g-1)$ for all $g \in G$; hence, $$\xi^{\mu}(g) = \int_{\AC} \tet(g-1)\; d\mu = \int_{\AC} \ev_{g}(\tet)\; d\mu,\qquad  g \in G.$$  It is clear that the function $\xi^{\mu}$ is constant on the superclasses of $G$, and that it is normalised (that is, $\xi^{\mu}(1) = 1$); thus, in order to conclude that $\xi^{\mu} \in \SCh(G)$, it remains to trove that $\xi^{\mu}$ is definite positive. We recall that, according to the Gelfand-Naimark-Segal construction, a function $\map{\vphi}{G}{\C}$ is a character if and only if there is a unitary representation $(\pi,\CH)$ and a cyclic vector $v \in \CH$ such that $$\vphi(g) = \langle \pi(g)v \mid v \rangle , \qquad g \in G,$$ where $\langle \cdot \mid \cdot \rangle$ denotes the inner product of $\CH$ (for details, we refer to \cite[Proposition~2.4.4]{Dixmier1977a}).

\begin{proposition} \label{PosDef}
For every $\mu \in \CM_{\G}^{+}(\AC)$, the function $\map{\xi^{\mu}}{G}{\C}$ is a supercharacter of $G$.
\end{proposition}

\begin{proof}
In order to see that $\xi^{\mu}$ is definite positive, we show that it is the character of $G$ afforded by an explicit cyclic representation of $G$. Let $\CH^{\mu}$ denote the Hilbert space $L^{2}(\AC,\mu)$; for every $f \in \CH^{\mu}$ and every $g \in G$, we define the linear operator $\map{\CT^{\mu}(g)}{\CH^{\mu}}{\CH^{\mu}}$ by $$(\CT^{\mu}(g)f)(\tet) = \ev_{g}(\tet) f(g\inv\tet),\qquad \tet \in \AC.$$ We claim that, for every $g \in G$, the operator $\CT^{\mu}(g)$ is unitary. To check this, let $g \in G$ be arbitrary, and compute the adjoint operator of $\CT^{\mu}(g)$: for every $f_{1}, f_{2} \in \CH^{\mu}$, we evaluate
\begin{align*}
\langle \CT^{\mu}(g)f_{1}\mid f_{2}\rangle &= \int_{\AC} \ev_{g}(\tet) f_{1}(g\inv\tet) \ovl{f_{2}(\tet)}\; d\mu = \int_{\AC}\ev_{g}(g\tet) f_{1}(\tet) \ovl{f_{2}(g\tet)}\; d\mu \\ &= \int_{\AC} f_{1}(\tet) \ovl{\ev_{g\inv}(\tet) f_{2}(g\tet)}\; d\mu = \langle f_{1}\mid \CT^{\mu}(g\inv)f_{2} \rangle
\end{align*}
(in the second equality, we took into account that the measure $\mu$ is $\G$-invariant), and thus the adjoint operator of $\CT^{\mu}(g)$ is $\CT^{\mu}(g\inv)$. On the other hand, we clearly have $\CT^{\mu}(g) \CT^{\mu}(g\inv) = \CT^{\mu}(g\inv)\CT^{\mu}(g) = \id_{\CH^{\mu}}$ where $\map{\id_{\CH^{\mu}}}{\CH^{\mu}}{\CH^{\mu}}$ denotes the identity operator, and hence $\CT^{\mu}(g)$ is unitary, as claimed. Furthermore, it is easy to see that the mapping $g \mapsto \CT^{\mu}(g)$ defines a unitary representation of $G$ on $\CH^{\mu}$.

Finally, since $\ev(G)$ is a dense subalgebra in $C(\AC)$, its image in $\CH^{\mu}$ is also dense; moreover, we have $\ev_{g} = \CT^{\mu}(g) \ev_{1}$ for all $g \in G$, and this implies that $\ev_{1}$ is a cyclic vector of $\CH^{\mu}$, and that the representation $(\CT^{\mu},\CH^{\mu})$ affords the character given by the formula $$\langle \CT^{\mu}(g)\ev_{1} \mid \ev_{1} \rangle = \int_{\AC}\ev_{g}(\tet)\; d\mu = \xi^{\mu}(g),\qquad g \in G,$$ as required.
\end{proof}

For every $\mu \in \CM_{\G}^{+}(\AC)$, we refer to the representation $(\CT^{\mu},\CH^{\mu})$, as defined in the previous proof, as the \textit{super-representation} of $G$ associated with $\mu$.

The following result is now a clear consequence of \refp{SCh0} (and also of Propositions~\ref{SCh} and \ref{PosDef}); for the last assertion, we note that a supercharacter $\xi \in \SCh(G)$ is indecomposable if and only if the corresponding $\G$-character $\zeta \in \Ch_{\G}(\CA)$ is indecomposable.

\begin{theorem} \label{ErgodicCorrespondence}
Let $G=1+\CA$ be a countable discrete algebra group (associated with a nil $\k$-algebra $\CA$). Then, the mapping $\xi \mapsto \mu^{\xi}$ defines an affine homeomorphism between $\SCh(G)$ and $\CM^{+}_{\G}(\AC)$ with inverse given by the mapping $\mu \mapsto \xi^{\mu}$. In particular, the indecomposable supercharacters of $G$ are in one-to-one correspondence with the ergodic measures in $\CM_{\G}^{+}(\AC)$.
\end{theorem}

Consequently, a parametrisation of the ergodic measures in $\CM_{\G}^{+}(\AC)$ yields a parametrisation of the set $\ISCh(G)$ consisting of all indecomposable supercharacters of the algebra group $G = 1 + \CA$. In the case where $G$ is finite, every $\G$-invariant measure on $\AC$ is supported on a unique $\G$-orbit and, conversely, every $\G$-orbit on $\AC$ supports a unique $\G$-invariant measure (hence, the $\G$-orbits on $\AC$ provide a complete description of $\ISCh(G)$. However, if $G = 1+\CA$ is infinite, then an orbit $\G\cdot \tet$ supports a $\G$-invariant measure if and only if $\G \cdot \tet$ is finite; we recall that the \textit{support} $\supp(\mu)$ of a probability measure $\mu$ on $\AC$ is defined to be the set consisting of all $\tet \in \AC$ such that $\mu(U) \in \R^{+}$ for every open neighbourhood $U \sset \AC$ of $\tet$ (equivalently, $\supp(\mu)$ the smallest closed subset $\CC$ of $\AC$ such that $\mu(\AC \setminus \CC) = 0$). Nevertheless, the following is true; henceforth, for every $\tet \in \AC$, we denote by $\CO^{\tet}$ the closure $\ovl{\G\cdot\tet}$ in $\AC$ of the $\G$-orbit $\G \cdot \tet$.

\begin{proposition} \label{EOrbital}
Let $G=1+\CA$ be a countable discrete algebra group (associated with a nil $\k$-algebra $\CA$). Then, for every ergodic measure $\mu \in \CM_{\G}^{+}(\AC)$, there is at least one $\tet \in \AC$ such that $\supp(\mu) = \CO^{\tet}$.
\end{proposition}

\begin{proof}
Let $\mu \in \CM_{\G}^{+}(\AC)$ be ergodic, and consider $\supp(\mu)$ equipped with the subspace topology; hence, in particular, $\supp(\mu)$ is second countable. Since $\supp(\mu)$ is clearly $\G$-invariant, $\mu$ can be though naturally as an ergodic $\G$-invariant measure on $\supp(\m)$ having full support. Let  $\set{U_{n}}{n \in \N}$ be a topological basis of $\supp(\mu)$; hence, $\mu(U_n) \in \R^{+}$ for all $n \in \N$. Since $\G \cdot U_{n}$ is a $\G$-invariant set of positive measure and $\mu$ is ergodic, we must have $\mu(\G \cdot U_n) = 1$; moreover, the family $\set{\G \cdot U_{n}}{n \in \N}$ is also a topological basis of $\supp(\mu)$. Let $V = \bigcap_{n \in \N} \G \cdot U_n$, and note that $\mu(V)=1$ (because $V$ is an intersection of sets with measure $1$), and this clearly implies that $V$ is non-empty; furthermore, for every $\tet \in V$ and every $n \in \N$, the intersection $\G \cdot \tet \cap U_{n}$ is non-empty. Since $\set{U_{n}}{n \in \N}$ be a topological basis of $\supp(\mu)$, we conclude that $\G \cdot \tet$ is a dense subset of $\supp(\m)$, which means that $\CO^{\tet} = \supp(\mu)$, as stated.
\end{proof}

In the case where the group $G$ is amenable, the previous result may be slightly improved; we recall that a discrete group $G$ is \textit{amenable} if it admits a \textit{F\o lner sequence}, that is, a family $\set{F_{n}}{n \in \N}$ of finite subsets of $G$ such that, for every $g \in G$, $$\lim_{n\to\infty} \frac{|F_{n} \bigtriangleup (gF_{n})|}{|F_{n}|} = 0$$ where $\bigtriangleup$ denotes the symmetric difference of sets (for the definition and a detailed exposition on amenable groups, we refer to \cite{Pier1984a}; see also \cite[Appendix~G]{Bekka2008a}). It is clear that $G$ is amenable if and only if $\G$ is amenable.

\begin{proposition}\label{AmenableOrbital}
Let $G=1+\CA$ be an amenable countable discrete algebra group (associated with a nil $\k$-algebra $\CA$). Then, for every $\tet \in \AC$, there is at least one ergodic measure $\mu \in \CM_{\G}^{+}(\AC)$ such that $\supp(\mu) = \CO^{\tet}$.
\end{proposition}

\begin{proof}
Let $\tet \in \AC$, and let $\set{F_{n}}{n \in \N}$ be a F\o lner sequence for $\G$; notice that $G$ is amenable if and only if $\G$ is amenable. Since $$\bigg|\frac{1}{|F_n|}\sum_{\g \in F_n}\ev_g(\g \cdot \tet) \bigg|\leq \frac{1}{|F_n|}\sum_{\g \in F_n} |\ev_g(\g \cdot \tet)| = 1,\qquad g \in G,$$ and since $\{\ev_g \colon g \in G \}$ is a countable dense subset of $C(\AC)$, \cite[Theorem~I.24]{Reed1980a} (applied twice) implies that there is a subsequence $\set{F_{n_{k}}}{k \in \N}$ of $\set{F_{n}}{n \in \N}$ such that the limit $$\lim_{k \to \infty}\frac{1}{|F_{n_{k}}|} \sum_{\g \in F_{n_{k}}} f(\g \cdot \tet)$$ exists for all $f \in C(\AC)$. Therefore, we may define a measure $\mu$ on $\AC$ by the rule $$\int_{\AC} f \; d\mu= \lim_{k \to \infty} \frac{1}{|F_{n_{k}}|}\sum_{\g \in F_{n_{k}}} f(\g \cdot \tet), \qquad f \in C(\AC);$$ since $\set{F_{n_{k}}}{k \in \N}$ is a F\o lner sequence for $\G$, it is straightforward to check that $\mu$ is a $\G$-invariant probability measure (that is, $\mu \in \CM_{\G}^{+}(\AC)$). Finally, for every Borel subset $B$ of $\AC$ with $\mu(B) \neq 0$, we have $$\mu(B) = \lim_{k \to \infty} \frac{1}{|F_{n_{k}}|} \big|\set{\g \in F_{n_{k}}}{\g \cdot \tet \in B } \big|,$$ and thus $\supp(\mu)$ is contained in the closure $\CO^{\tet}$ of $\G \cdot \tet$. Furthermore, if $B$ is a $\G$-invariant Borel subset of $\AC$ with $\mu(B) > 0$, then $B \cap \G \cdot \tet$ is non empty, and thus $\G \cdot \tet \subseteq B$, which implies that $\set{\g \in F_{n_{k}}}{\g \cdot \tet \in B} = F_{n_{k}}$ for all $k \in \N$, and hence $\mu(B)=1$. It follows that the measure $\mu \in \CM_{\G}^{+}(\AC)$ is ergodic, and this concludes the proof.
\end{proof}

We remark that, in the case where $G = 1 + \CA$ is a finite algebra group, the previous results assure that every indecomposable supercharacter of $G$ correspond uniquely to the closure of some $\G$-orbit on $\AC$ (see \refeq{FSCh}); however, it is not clearly deduced that a similar result holds in the more general case of an infinite amenable discrete countable algebra group (in fact, the closure of a given $\G$-orbit may support different ergodic $\G$-invariant measures, and hence the the indecomposable supercharacters of $G$ are not necessarily parametrised by the closures of the $\G$-orbits on $\AC$).

\section{The regular representation} \label{RegSCh}

As before, we let $G=1+\CA$ be a discrete countable algebra group associated with a nil algebra $\CA$ over a (countable discrete) field $\k$, and consider the usual \textit{left regular representation} $(\pi,L^{2}(G,\kappa))$ of $G$ on the Hilbert space $L^{2}(G,\kappa)$ where $\kappa$ is the counting measure on $G$ (which can be taken as a Haar measure on $G$); recall that, for every $g \in G$ and every $f \in L^2(G,\kappa)$, the map $\pi(g)f \in L^2(G,\kappa)$ is defined by $\pi(g)f(x)=f(g^{-1}x)$ for all $x \in G$. For every $g \in G$, let $\del_g \in L^2(G,\kappa)$ be the Dirac function supported on $\{g\}$, and note that the $\C$-linear span of the set $\set{\del_{g}}{g \in G}$ is dense in $L^2(G,\kappa)$. Since $\pi(g)\del_{1} = \del_{g}$ for all $g \in G$, we see that the function $\del_{1}$ is a cyclic vector for the representation $(\pi,L^2(G,\kappa))$, and thus the regular representation affords a character of $G$ to which we refer as the \textit{regular character} of $G$ and denote by $\lam_{G}$; hence, $\lam_{G}(g) = \langle \pi(g)\del_{1} \mid \del_{1} \rangle = \del_{g,1}$ for all $g \in G$. It is clear that $\lam_{G}$ is a supercharacter of $G$, and so in virtue of \reft{ErgodicCorrespondence} it is uniquely determined by a unique $\G$-invariant measure on $\AC$. In what follows, we characterise $\lam_{G}$ as a supercharacter of $G$ by understanding the measure on $\AC$ which is associated with $\lam_{G}$; in particular, we provide a criterion for the regular character to be an indecomposable supercharacter (or equivalently for the corresponding $\G$-invariant measure to be ergodic), and determine its integral decomposition over $\ISCh(G)$.

For the moment, we consider a slightly more general situation where $\CA = \GC$ is the Pontryagin dual of an arbitrary compact abelian group $\CG$; hence, $\CA$ is a discrete abelian group, and it is countable if and only if $\CG$ is second countable (or, equivalently, metrisable). Since $\CG$ is a compact group, there is a unique probability Haar measure on $\CG$ which we denote by $\eta$ (we recall that $\CG$ is unimodular, and thus the left and right Haar measures coincide); on the other hand, we denote by $\kappa$ the counting measure on $\GC$ (which is a Haar measure on $\GC$; notice that, since $\GC$ abelian, it is trivially unimodular). For every $\vphi \in L^{1}(\CG,\eta)$, we define the \textit{Fourier transform} of $\vphi$ to be the continuous map $\map{\widehat{\vphi}}{\GC}{\C}$ given by $$\widehat{\vphi}(\tet) = \int_{\CG} \vphi(g)\overline{\tet(g)}\; d\eta,\qquad \tet \in \GC;$$ for details on the Fourier transform on arbitrary locally compact abelian groups, we refer to \cite[Chapter~4]{Folland1995a}. In the case where $\widehat{\vphi} \in L^{1}(\GC,\kappa)$ for $\vphi \in L^{1}(\CG,\eta)$, the \textit{Fourier inversion formula} holds (see \cite[Theorem~4.32]{Folland1995a}; see also \cite[Proposition~4.24]{Folland1995a}): for $\eta$-almost every $g \in \CG$,  we have $$\vphi(g) = \int_{\GC} \widehat{\vphi}(\tet) \tet(g)\;d\kappa;$$ indeed, this formula holds for all $g \in \CG$ whenever $\vphi$ is continuous. Finally, we mention \textit{Plancherel theorem} (see \cite[Theorem~4.32]{Folland1995a}) which asserts that, when restricted to $L^{1}(\CG,\eta) \cap L^{2}(\CG,\eta)$, the Fourier transform extends uniquely to a uni\-ta\-ry isomorphism $\map{\CF}{L^{2}(\CG,\eta)}{L^{2}(\GC,\kappa)}$ (so that the Fourier transform is defined in the space $L^{1}(\CG,\eta)+L^{2}(\CG,\eta)$).

Now, let $\G$ be a group which acts on the left of $\GC$ via continuous automorphisms; as before, we write $\g \cdot a = \g(a)$ for all $\g \in \G$ and all $a \in \GC$. Then, $\G$ acts continuously on the left of $\CG$ via the contragradient action, and thus every $\g \in \G$ defines a measure $\g \cdot \eta$ on $\CG$ by the rule $(\g \cdot \eta)(B) = \eta(\g\inv\cdot B)$ for all Borel subset $B$ of $\CG$; notice that $(\g \cdot \eta)(\CG) = \eta(\CG) = 1$, and hence $\g \cdot \eta$ is a probability measure. It is well-known that, since every element $\g \in \G$ is a continuous automorphism of $\CG$ with continuous inverse, the composition $\eta \circ \g$ is also a (left) Haar measure on $\CG$, and thus the Haar measure $\eta$ on $\CG$ is $\G$-invariant. As it turns out, the ergodicity of the Haar measure $\eta$ with respect to the $\G$-action on $\CG$ depends on the existence of finite $\G$-orbits on $\GC$. Indeed, the following result is an obvious generalisation of a certainly well-known theorem on classical ergodic theory on compact groups (see, for example, \cite[Theorem~1.10]{Walters1982a}); a proof is included for convenience of the reader.

\begin{theorem}\label{FiniteSCl}
Let $\CG$ be a compact abelian group equipped with the normalised Haar measure $\eta$, and let $\G$ be a group of continuous automorphisms of $\CG$. Then, $\eta$ is ergodic (with respect to the $\G$-action on $\CG$) if and only if every nontrivial character $\tet \in \GC$ has an infinite $\G$-orbit.
\end{theorem}

\begin{proof}
Firstly, we assume that $\eta$ is ergodic, and let $\tet \in \GC$ be such that the $\G$-orbit $\G \cdot \tet \sset \GC$ is finite. Then, since the character $\vphi_{\G\cdot\tet} = \sum_{\tau \in \G\cdot\tet} \tau$ of $\CG$ is clearly $\G$-invariant, the ergodicity of $\eta$ implies that there is a constant $c$ such that $\vphi_{\G\cdot\tet}(g) = c$ for $\eta$-almost all $g \in \CG$, and thus $$c = \int_{\CG} \vphi_{\G\cdot\tet}(g)\; d\eta = \langle \vphi_{\G\cdot\tet}, \1_{\CG} \rangle_{\CG} = |\G\cdot\tet|\, \langle \tet, \1_{\CG} \rangle_{\CG}$$ where $\langle \cdot, \cdot \rangle_{\CG}$ denotes the usual  inner product on $L^{2}(\CG,\eta)$, and where $\1_{\CG} \in \GC$ is the trivial character of $\CG$. By the orthonormality of the indecomposable characters of $\CG$, we conclude that $c = 0$ unless $\tet = \1_{\CG}$, in which case $\G\cdot \tet = \{\1_{\CG}\}$; on the other hand, if $c = 0$, then we obtain $$0 = \int_{\CG} \vphi_{\G\cdot\tet}(g) \ovl{\tet(g)}\; d\eta = \langle \vphi_{\G\cdot\tet}, \tet \rangle_{\CG} = 1,$$ a contradiction. Consequently, $\{\1_{\CG}\}$ is the unique finite $\G$-orbit on $\GC$, as required.

Conversely, suppose that $\{\1_{\CG}\}$ is the unique finite $\G$-orbit on $\GC$. In order to prove that $\eta$ is an ergodic $\G$-invariant measure, we show that every $\G$-invariant function in $L^{2}(\CG,\eta)$ is constant $\eta$-almost everywhere. Let $\vphi \in L^{2}(\CG,\eta)$ be $\G$-invariant, and consider the Fourier transform $\widehat{\vphi} = \CF(\vphi) \in L^{2}(\GC,\kappa)$. Since the Haar measure $\eta$ is $\G$-invariant, it is straightforward to check that $\widehat{\vphi}(\g\cdot\tet) = \widehat{\vphi}(\tet)$ for all $\g \in \G$ and all $\tet \in \GC$ (which means that $\widehat{\vphi}$  is $\G$-invariant), and thus $$\int_{\G\cdot\tet}  |\widehat{\vphi}(\tet')|^{2}\;d\kappa = \kappa(\G\cdot\tet)\,|\widehat{\vphi}(\tet)|^{2}.$$ Since $\widehat{\vphi} \in L^{2}(\GC,\kappa)$, we conclude that $\widehat{\vphi}(\tet) = 0$ whenener $\tet \in \GC$ has an infinite $\G$-orbit (recall that $\GC$ is discrete, and that $\kappa$ is the counting measure on $\GC$). It follows that $\widehat{\vphi} = c \1_{\CG}$ where $c = \widehat{\vphi}(\1_{\CG})$, and so the Fourier inversion formula implies that $\vphi$ is $\eta$-almost everywhere constantly equal to $c$, as required.
\end{proof}

We now return to the case where $G=1+\CA$ is an arbitrary countable discrete algebra group (associated with a nil $\k$-algebra $\CA$) and $\G=G \times G$. Since the normalised Haar measure $\eta$ on $\AC$ is $\G$-invariant, it uniquely defines a supercharacter $\xi^{\eta} \in \SCh(G)$ by the rule $$\xi^{\eta}(g) = \int_{\AC}\tet(g-1)\; d\eta = \del_{g,1},\qquad g \in G;$$ hence, $\xi^{\eta}$ is the regular character $\lam_{G}$ of $G$. Let $\CH^{\eta} = L^2(\AC,\eta)$, and consider the super-representation $(\CT^{\eta},\CH^{\eta})$ of $G$. Since $(\CT^{\eta},\CH^{\eta})$ affords the regular character of $G$, it follows that $(\CT^{\eta},\CH^{\eta})$ is quasi-equivalent to the regular representation $(\pi,L^{2}(G,\kappa))$; in fact, these representations are equivalent.

\begin{proposition} \label{regular}
For every countable discrete algebra group $G=1+\CA$ (associated with a nil $\k$-algebra $\CA$), the linear operator $\map{\CL}{\CH^{\eta}}{L^{2}(G,\kappa)}$ defined by $$\CL(\bet)(g) = \int_{\AC} \bet(\tet) \overline{\tet(g-1)}\; d \eta, \qquad \bet \in \CH^{\eta},\ g \in G,$$ defines an invertible intertwining operator between the representations $(\CT^{\eta},\CH^{\eta})$ and $(\pi,L^{2}(G,\kappa))$ of $G$ whose inverse $\map{\CL\inv}{L^{2}(G,\kappa)}{\CH^{\eta}}$ is defined on the functions with finite support $\varsigma \in C_{c}(G)$ by the rule $$\CL^{-1}(\varsigma)(\tet) = \sum_{g \in G} \varsigma(g) \tet(g-1), \qquad \tet \in \AC.$$ \end{proposition}

\begin{proof}
Firstly, we note that the map $\map{\nu}{G}{\CA}$ (defined by $\nu(g) = g-1$ for all $g \in G$) induces a unitary isomorphism of Hilbert spaces $\map{\nu^{\ast}}{L^{2}(G,\kappa)}{L^{2}(\CA,\kappa_{0})}$. Furthermore, we have $\CL = \CF\inv \circ \nu^{\ast}$, and hence it only remains to show that $\CL$ is an intertwining operator. To see this, for every $g,h \in G$ and every $\bet \in \CH^\eta$, we evaluate
\begin{align*}
\CL\big(\CT^\eta(g)\bet\big)(h) &= \int_{\AC} \big(\CT^\eta(g) \bet \big)(\tet)\, \overline{\tet(h-1)} \; d\eta = \int_{\AC} \bet(g\inv\tet)\, \ev_{g}(\tet)\, \tet(1-h) \; d\eta \\ &= \int_{\AC} \bet(g\inv\tet)\, \tet(g-1)\, \tet(1-h) \; d\eta = \int_{\AC} \bet(g\inv\tet)\, \tet(g-h) \; d\eta \\ &= \int_{\AC} \bet(\tet)\, (g\tet)(g-h)\; d\eta = \int_{\AC} \bet(\tet)\, \tet(1-g\inv h)\; d\eta \\ &= \int_{\AC} \bet(\tet)\, \overline{\tet(g\inv h-1)}\; d\eta = \big(\pi(g)\CL(\bet)\big)(h)
\end{align*}
where the fourth equality holds by the $G$-invariance of $\eta$. The result follows.
\end{proof}

Consequently, the regular representation of $G = 1+\CA$ is completely determined by the normalised Haar measure $\eta$ on the dual group $\AC$; in particular, the regular character $\lam_{G}$ is an indecomposable supercharacter if and only if $\eta$ is an ergodic $\G$-invariant measure. Since $\supp(\eta) = \AC$, \refp{EOrbital} imply the following immediate corollary to the previous proposition.

\begin{theorem} \label{FSclErgodic}
Let $G$ be a countable discrete algebra group associated with a nil $\k$-algebra $\CA$. Then, the regular character $\lam_{G}$ is an indecomposable supercharacter of $G$ if and only if $\{ 1 \}$ is the unique finite superclass.
\end{theorem}

Therefore, we see that the nature of the regular character, as an indecomposable supercharacter, is not intrinsic to the class of countable discrete algebra groups, but rather to the nature of the $\G$-action on $\AC$ (or, equivalently, on the $\G$-action on $\CA$).

\begin{corollary}
Let $G=1+\CA$ be a countable discrete algebra group associated with a nil $\k$-algebra $\CA$, and suppose that $\CA$ is finite dimensional. Then, the regular character $\lam_{G}$ is not an indecomposable supercharacter of $G$.
\end{corollary}

\begin{proof}
Every finite dimensional nil algebra is a nilpotent algebra, and thus there is a non-zero element $a \in \CA$ satisfying $ab=ba=0$ for all $b \in \CA$. It is obvious that the $\{1+a\}$ is a finite superclass of $G$.
\end{proof}

By the way of example, if $\k$ be an arbitrary countable infinite discrete field, then the regular character of the infinite unitriangular group $U_{n}(\k)$ is not an indecomposable supercharacter; indeed, $U_{n}(\k) = 1 + \fru_{n}(\k)$ and $\fru_{n}(k)$ is finite-dimensional (nilpotent) $\k$-algebra. On the other extreme, every non-trivial superclass of the locally finite unitriangular group $U_{\infty}(\k)$ is infinite, and hence the regular character of $U_{\infty}(\k)$ is an indecomposable supercharacter.

At this point, it is worth to mention that $\lam_{G}$ is an indecomposable character of $G$ if and only if  the unique finite conjugacy class is the trivial conjugacy class $\{ 1 \}$ (see \cite[Lemma~5.3.4]{Murray1943a}); hence, \reft{FSclErgodic} may be though as a supercharacter  analogue of this result.

As shown by Thoma (see \cite{Thoma1964a,Thoma1967a}; see also \cite{Bekka2020b}), the connection between the regular character of an arbitrary countable discrete group $G$ and its finite conjugacy classes is much deeper. Indeed, it is possible to define a Plancherel formula determined by the normal subgroup $G_{\fc}$ consisting of all elements of $G$ with finite conjugacy class. The conjugation action of $G$ on $G_{\fc}$ induces a natural action on the set of characters $\Ch(G_{\fc})$ of $G_{\fc}$; let $\Ch_{G}(G_{\fc})$ denote the convex subset of $\Ch(G_{\fc})$  consisting of all $G$-invariant elements of $\Ch(G_{\fc})$, and let $\ICh_{G}(G_{\fc})$ denote the corresponding set of extreme (or indecomposable) elements. Given an arbitrary $\vphi \in \ICh_{G}(G_{\fc})$ we denote by $\widetilde{\vphi}$ its extension by zero of $\vphi$ to $G$; hence, for every $g \in G$, $$\widetilde{\vphi}(g) = \begin{cases} \vphi(g), & \text{if $g \in G_{\fc}$,} \\ 0, & \text{if $g \in G\setminus G_{\fc}$.} \end{cases}$$ It is straightforward to check that $\widetilde{\vphi} \in \Ch(G)$ for all $\vphi \in \Ch(G_{\fc})$; indeed, as proved in \cite[Proposition~1.F.9]{Bekka2020c} (see also \cite{Blattner1963a}, or \cite[Theorem~6.13]{Folland1995a}), if $(\pi_{\vphi},\CH_{\vphi})$ is the GNS representation of $G_{\fc}$ associated with $\vphi$, then the GNS representation $(\pi_{\widetilde{\vphi}},\CH_{\widetilde{\vphi}})$ of $G$ associated with $\widetilde{\vphi}$ is equivalent to the induced representation $\Ind^{G}_{G_{\fc}}(\pi_{\vphi})$, and thus there is a cyclic vector $\upsilon \in \CH_{\widetilde{\vphi}}$ such that $$\widetilde{\vphi}(g) = \big\langle \Ind^{G}_{G_{\fc}}(\pi_{\vphi})(g) \upsilon, \upsilon \big\rangle,\qquad g \in G.$$ If we set $$\widetilde{\Ch}{}^{+}_{G}(G_{\fc}) = \set{\widetilde{\vphi}}{\vphi \in \ICh_{G}(G_{\fc})},$$ then \cite[Satz~4]{Thoma1967a} asserts that the intersection $\widetilde{\Ch}{}^{+}_{G}(G_{\fc}) \cap \ICh(G)$ is non-empty; moreover, if $\omega$ denotes the Choquet measure on $\ICh(G)$ which is associated with the regular character $\lam_{G}$, then $\supp(\omega) = \widetilde{\Ch}{}^{+}_{G}(G_{\fc})$, and thus we have
\begin{equation} \label{regular1}
\lam_{G}(g) = \int_{\widetilde{\Ch}{}^{+}_{G}(G_{\fc})} \chi(g) \;d\omega,\qquad g \in G.
\end{equation}
Since the regular character $\lam_{G}$ is afforded by the regular representation of $G$ and since $\ICh(G)$ can be understood as a dual space for $G$, it is somewhat customary to refer to $\omega$ as the \textit{Plancherel measure} of $G$; more details about the regular character of nilpotent discrete groups can be found in \cite{Baggett1997a}. (Recent developments on the Plancherel formula for countable groups can be found in \cite{Bekka2020b}.)

As it turns out, there is also a counterpart for this result in terms of finite superclasses of a countable discrete algebra group. Let $G=1+\CA$ be an arbitrary countable discrete algebra group (associated with a nil $\k$-algebra $\CA$), and let $\nu$ be the unique probability measure on $\ISCh(G)$ such that
\begin{equation} \label{regular2}
\lambda_{G}(g)=\int_{\ISCh(G)}\xi(g)\; d\nu,\qquad g \in G;
\end{equation}
in direct analogy with the integral decomposition of $\lambda_{G}$ with respect to indecomposable characters, we refer to the measure $\nu$ as the \textit{super-Plancherel measure}. Our aim is to describe this measure $\nu$ in terms of the supercharacters of the algebra subgroup  $G_{\fsc}$ of $G$ consisting of all elements having a finite superclass; we start by proving that $G_{\fsc}$ is indeed an algebra subgroup of $G$.

\begin{lemma} \label{afsc}
In the above notation, let $\CA_{\fsc}$ be the subset of $\CA$ consisting of all elements with finite $\G$-orbit. Then, $\CA_{\fsc}$ is a (nil) subalgebra of $\CA$ and $G_{\fsc} = 1+\CA_{\fsc}$. In particular, $G_{\fsc}$ is an algebra subgroup of $G$; furthermore, $G_{\fsc}$ is a normal subgroup of $G$.
\end{lemma}

\begin{proof}
It is obvious that $\CA_{\fsc}$ is closed under scalar multiplication. On the other hand, if $a,b \in \CA_{\fsc}$, then $\G \cdot(a+b) \sset \G\cdot a + \G\cdot b$, and hence
$\G \cdot (a+b)$ is a finite $\G$-orbit, which means that $a+b \in \CA_{\fsc}$. Finally, if $a,b \in \CA_{\fsc}$, then $$\G \cdot (ab) = \set{gag\inv gbh}{g,h \in G} \sset \set{gag\inv}{g \in G}\,(\G\cdot b) \sset (\G\cdot a)\, (\G\cdot b),$$ and so the $\G$-orbit $\G\cdot(ab)$ is finite, which means that $ab \in \CA_{\fsc}$. The result is now clear.
\end{proof}

The group $\G$ acts continuously on the left of $\CA_{\fsc}$, and thus we naturally obtain a continuous left action of $\G$ on the Pontryagin dual $(\CA_{\fsc})^{\circ}$ of the additive group $(\CA_{\fsc})^{+}$. Following the terminology used in \refp{SCh0}, let $\Ch_{\G}(\CA_{\fsc})$ denote the set of all $\G$-invariant characters of $(\CA_{\fsc})^{+}$, and let $\ICh_{\G}(\CA_{\fsc})$ denote the set consisting of all indecomposable $\G$-invariant characters of $(\CA_{\fsc})^{+}$. By \refp{SCh0}, we know that $\Ch_{\G}(\CA_{\fsc})$ is affinelly homeomorphic to the space $\CM^{+}_{\G}((\CA_{\fsc})^{\circ})$ of $\G$-invariant probability measures on $(\CA_{\fsc})^{\circ}$, so that the subset $\ICh_{\G}(\CA_{\fsc})$ of indecomposable $\G$-supercharacters is in one-to-one correspondence with the set of ergodic measures in $\CM^{+}_{\G}((\CA_{\fsc})^{\circ})$.

On the other hand, if we set $\G_{\fsc} = G_{\fsc}\x G_{\fsc}$, then \reft{ErgodicCorrespondence} asserts that the set $\SCh(G_{\fsc})$ of supercharacters of $G_{\fsc}$ is affinelly homeomorphic to the space $\CM^{+}_{\G_{\fsc}}((\CA_{\fsc})^{\circ})$ of $\G_{\fsc}$-invariant probability measures on $(\CA_{\fsc})^{\circ}$. In particular, since we clearly have $\CM^{+}_{\G}((\CA_{\fsc})^{\circ}) \sset \CM^{+}_{\G_{\fsc}}((\CA_{\fsc})^{\circ})$, we conclude that every $\G$-invariant probability measure $\mu \in \CM^{+}_{\G}((\CA_{\fsc})^{\circ})$ corresponds uniquely to the supercharacter $\xi^{\mu} \in \SCh(G_{\fsc})$ defined by the rule $$\xi^{\mu}(g) = \int_{(\CA_{\fsc})^{\circ}} \tet(g-1)\; d\mu,\qquad  g \in G_{\fsc}.$$ Henceforth, we define $$\SCh_{\G}(G_{\fsc}) = \set{\xi^{\mu}}{\mu \in \CM^{+}_{\G}((\CA_{\fsc})^{\circ})},$$ and denote by $\ISCh_{\G}(G_{\fsc})$ the subset of $\SCh_{\G}(G_{\fsc})$ consisting of all supercharacters of $G_{\fsc}$ which correspond to the ergodic $\G$-invariant measures in $\CM^{+}_{\G}((\CA_{\fsc})^{\circ})$; furthermore, for every $\xi \in \SCh_{\G}(G_{\fsc})$, we denote by $\widetilde{\xi}$ the extension by zero of $\xi$ to $G$. Our next goal is to prove the following result.

\begin{theorem} \label{FteSch}
Let $G = 1+\CA$ be a discrete countable algebra group associated with a nil $\k$-algebra $\CA$, and let $ \xi \in \SCh^{+}_{\G}(G_{\fsc})$ be arbitrary. Then, $\widetilde{\xi}$ is an indecomposable supercharacter of $G$.
\end{theorem}

In order to proceed with the proof of this theorem, we consider a slightly more general situation where $\CB$ is an arbitrary $\G$-invariant subalgebra of $\CA$; we consider the Pontryagin dual $\CB^{\circ}$ of the additive group $\CB^{+}$, and the natural contragradient $\G$-action on $\CB^{\circ}$. By \refp{SCh0}, we know that the space $\Ch_{\G}(\CB)$ consisting of all $\G$-characters of $\CB^{+}$ is affinelly homeomorphic to the space $\CM^{+}_{\G}(\CB^{\circ})$ of $\G$-invariant probability measures on $\CB^{\circ}$, and so the subset $\ICh_{\G}(\CB)$ consisting of all indecomposable $\G$-characters of $\CB^{+}$ is in one-to-one correspondence with the set of ergodic measures in $\CM^{+}_{\G}(\CB^{\circ})$. On the other hand, if we set $\Lambda = H\x H$, then \reft{ErgodicCorrespondence} asserts that \reft{ErgodicCorrespondence} the space $\SCh(H)$ of supercharacters of the algebra group $H = 1+\CB$ (which is countable and discrete) is affinelly homeomorphic to the space $\CM^{+}_{\Lambda}(\CB^{\circ})$ of $\Lambda$-invariant probability measures on $\CB^{\circ}$. In particular, since we clearly have $\CM^{+}_{\G}(\CB^{\circ}) \sset \CM^{+}_{\Lambda}(\CB^{\circ})$, we conclude that every $\G$-invariant probability measure $\mu \in \CM^{+}_{\G}(\CB^{\circ})$ corresponds uniquely to the supercharacter $\xi^{\mu} \in \SCh(H)$ defined by the rule $$\xi^{\mu}(h) = \int_{\CB^{\circ}} \tet(h-1)\; d\mu,\qquad  h \in H.$$ Henceforth, we set $$\SCh_{\G}(H) = \set{\xi^{\mu}}{\mu \in \CM^{+}_{\G}(\CB^{\circ})},$$ and denote by $\ISCh_{\G}(H)$ the subset of $\SCh_{\G}(H)$ consisting of all supercharacters of $H$ which correspond to the ergodic $\G$-invariant measures in $\CM^{+}_{\G}(\CB^{\circ})$; furthermore, for every $\xi \in \SCh_{\G}(H)$, we let $\widetilde{\xi}$ denote the trivial extension by zero of $\xi$ to $G$.

Now, let $\CB'$ be an additive subgroup of $\CA$ such that $\CA$ decomposes as the direct sum $\CA = \CB \oplus \CB'$; notice that, since $\CB$ is a $\k$-linear subspace of $\CA$, it admits a $\k$-linear complement $\CB'$ (which is obviously an additive subgroup of $\CA$). Furthermore, let $\CB^{\perp}=\set{\tet \in \AC}{\CB \sset \ker(\tet)}$ be the closed subgroup of $\AC$ which is orthogonal to $\CB$, and note that $\CB^{\perp}$ is homeomorphic to the Pontryagin dual of $\CB'$; indeed, $\CB'$ is a canonically isomorphic to the quotient group $\CA\slash\CB$, and it is well-known that the Pontryagin dual $(\CA\slash\CB)^{\circ}$ of $\CA\slash\CB$ is naturally isomorphic to $\CB^{\perp}$ (see \cite[Theorem~4.39]{Folland1995a}). It follows that $\AC$ is homeomorphic to the product space $\CB^{\circ}\x \CB^{\perp}$; since $\CB^{\circ}$  is metrisable (because $\CB$ is countable), \cite[Corollary to Theorem~8.1]{Johnson1966a} guarantees that every Borel subset of $\AC$ is of the form $B_{1} \x B_{2}$ where $B_{1}$ is a Borel subset of $\CB^{\circ}$ and $B_{2}$ is a Borel subset of $\CB^{\perp}$. Therefore, every $\G$-invariant probability measure $\nu \in \CM_{\G}^{+}(\AC)$ is uniquely factorised as a product measure $\nu = \nu_{1} \x \nu_{2}$ where $\nu_{1} \in \CM_{\G}^{+}(\CB^{\circ})$ and $\nu_{2} \in \CM_{\G}^{+}(\CB^{\perp})$, and hence the associated supercharacter $\xi^{\nu} \in \SCh(G)$ factorises as follows: if $a \in \CA$ and $a = b+b'$ for $b \in \CB$ and $b' \in \CB'$, then $$\xi^{\nu}(1+a) = \int_{\AC} \tau(a)\; d\nu = \bigg(\int_{\CB^{\circ}} \tau(b)\; d\nu_{1} \bigg)\bigg(\int_{\CB^{\perp}}\tau'(b')\; d\nu_{2} \bigg).$$ We are now able to prove the following auxiliary result.

\begin{proposition}\label{Extension}
Let $G = 1+\CA$ be a discrete countable algebra group associated with a nil $\k$-algebra $\CA$, let $\CB$ be a $\G$-invariant subalgebra of $\CA$, and let $H = 1+\CB$. Moreover, let $\xi \in \SCh_{\G}(H)$, let $\widetilde{\xi} \in \SCh(G)$ be the trivial extension by zero of $\xi$ to $G$, and let $\widetilde{\mu}\in \CM_{\G}^{+}(\AC)$ be the $\G$-invariant probability measure associated with $\widetilde{\xi}$. Then, $\widetilde{\mu}$ factorises uniquely as the product $\widetilde{\mu} = \mu \x \eta_{\CB^{\perp}}$ where $\mu \in \CM_{\G}^{+}(\CB^{\circ})$ is the $\G$-invariant measure associated with $\xi$ and $\eta_{\CB^{\perp}}$ is the normalised Haar measure of $\CB^{\perp}$. In particular, $\widetilde{\xi}$ is an indecomposable supercharacter of $G$ if and only if both measures $\mu$ and $\eta_{\CB^{\perp}}$ are ergodic with respect to the $\G$-action on $\CB^{\circ}$ and on $\CB^{\perp}$, respectively.
\end{proposition}

\begin{proof}
The proof is a matter of straightforward calculations; notice that, since $\CB$ is $\G$-invariant, $\CB^{\perp}$ is a $\G$-invariant compact subgroup of $\AC$, and hence the Haar measure $\eta_{\CB^{\perp}}$ is $\G$-invariant. Indeed, let $a \in \CA$ be arbitrary, and decompose $a = b+b'$ where $b \in \CB$ and $b' \in \CB'$ are unique. Then, $$\int_{\AC}\tet(a)\; d(\mu\x \eta_{\CB^{\perp}}) = \bigg(\int_{\CB^{\circ}}\tau(b)\; d\mu \bigg) \bigg(\int_{\CB^{\perp}}\tau'(b')\; d\eta_{\CB^{\perp}} \bigg);$$ since $\CB^{\perp}$ is compact, the orthogonality relations imply that $\int_{\CB^{\perp}}\tau'(b')\; d\eta_{\CB^{\perp}} = \del_{b',0}$ for all $b' \in \CB'$, and thus $$\int_{\AC}\tet(a)\; d(\mu\x \eta_{\CB^{\perp}}) = \begin{cases} \xi(1+a), & \text{if $a \in \CB$,} \\ 0, & \text{otherwise,} \end{cases}$$ and hence $$\widetilde{\xi}(g)=\int_{\AC}\tet(g-1)\; d(\mu\x \eta_{\CB^{\perp}}).$$ By the unicity of the measure $\widetilde{\mu}\in \CM_{\G}^{+}(\AC)$ (see \reft{ErgodicCorrespondence}), we conclude that $\widetilde{\mu}=\mu \times \eta_{\CB^{\perp}}$, and this completes the proof (the remaining assertions are clear).
\end{proof}

Finally, we consider the regular character $\lambda_{G}$ of the countable discrete algebra group $G=1+\CA$. On the one hand, we recall from \refeq{regular1} that $$\lambda_{G}(g) = \int_{\widetilde{\Ch}^{+}_{G}(G_{\fc})}\chi(g)\; d\omega,\qquad g \in G,$$ where $\omega$ is the Choquet measure on $\ICh(G)$ associated with $\lam_{G}$; on the other hand, since $\lambda_{G} \in \SCh(G)$, we know that there is a unique probability measure $\nu$ on $\ISCh(G)$ such that $$\lambda_{G}(g)=\int_{\ISCh(G)}\xi(g)\; d\nu,\qquad g \in G$$ (see \refeq{regular1}). By comparing these two integral decompositions, we conclude that $\xi(g)=0$ for all $g \in G\setminus G_{\fc}$ and $\nu$-almost all $\xi \in \ISCh(G)$. We consider the adjoint (left) action of $G$ on $\CA$ given by $g\cdot a=gag^{-1}$ for all $g \in G$ and all $a \in \CA$, and denote by $\CA_{\fc}$ the subset of $\CA$ consisting of all elements with finite $G$-orbit (hence, $G_{\fc}=1+\CA_{\fc}$). Moreover, we define $$\CA_{\fc}(\G) = \set{a \in \CA_{\fc}}{\G \cdot a \sset \CA_{\fc}};$$ we note that, if $a \in \CA\setminus \CA_{\fc}(\G)$, then there is an element $\g \in \G$ such that $\g \cdot a \notin \CA_{\fc}$ and $\xi(1+a)=\xi(1+\g \cdot a)$ for all $\xi \in \ISCh(G)$, and thus $\xi(1+a)=0$ for all $a \notin \CA_{\fc}(\G)$ and $\nu$-almost all $\xi \in \ISCh(G)$. Using an argument similar to the proof of \refl{afsc}, it is easy to see that both $\CA_{\fc}$ and $\CA_{\fc}(\G)$ are subalgebras of $\CA$; it is also obvious that $\CA_{\fsc} \sset \CA_{\fc}(\G) \sset \CA_{\fc}$.

Let $\xi \in \ISCh(G)$ be such that $\xi(1+a) = 0$ for all $a \in \CA\setminus \CA_{\fc}(\G)$, and $\xi_{0}$ denote the restriction to $G_{\fc}(\G) = 1+\CA_{\fc}(\G)$; notice that $\xi_{0} \in \ISCh_{\G}(G_{\fc}(\G))$ and that the trivial extension by zero of $\xi_{0}$ to $G$ equals $\xi$, and thus \refl{Extension} implies that the normalised Haar measure $\eta_{\CA_{\fc}(\G)^{\perp}}$ of $\CA_{\fc}(\G)^{\perp}$ is ergodic with respect to the $\G$-action.

We are now able to conclude the proof of \reft{FteSch}.

\begin{proof}[Proof of \reft{FteSch}]
It is clear that $\widetilde{\xi}$ is a supercharacter of $G$, and thus \reft{ErgodicCorrespondence} guarantees there is a unique $\G$-invariant probability measure $\widetilde{\mu} \in \CM^{+}_{\G}(\AC)$ such that $$\widetilde{\xi}(g) = \int_{\AC} \tet(g-1)\; d\widetilde{\mu},\qquad  g \in G;$$ moreover, we know that $\widetilde{\xi} \in \ISCh(G)$ if and only if $\widetilde{\mu}$ is ergodic. According to \refp{Extension}, $\widetilde{\mu}$ factorises uniquely as the product $\widetilde{\mu} = \mu \x \eta_{(\CA_{\fsc})^{\perp}}$ where $\mu \in \CM_{\G}^{+}((\CA_{\fsc})^{\circ})$ is the $\G$-invariant measure associated with $\xi$ and $\eta_{(\CA_{\fsc})^{\perp}}$ is the normalised Haar measure of $(\CA_{\fsc})^{\perp}$; since $ \xi \in \SCh^{+}_{\G}(G_{\fsc})$, we know that the measure $\mu$ is ergodic, and thus $\widetilde{\mu}$ is ergodic if and only if $\eta_{(\CA_{\fsc})^{\perp}}$ is ergodic (with respect to the $\G$-action). Finally, since $\CA_{\fc}(\G)^{\perp}$ is homeomorphic to the quotient $\AC \slash \CA_{\fc}(\G)^{\circ}$ and since there is a natural homeomorphism $$\AC \slash \CA_{\fc}(\G)^{\circ} \cong \AC\slash (\CA_{\fsc})^{\circ} \times (\CA_{\fsc})^{\circ}\slash (\CA_{\fc}(\G)^{\circ},$$ we conclude that the measure $\eta_{\CA_{\fc}(\G)^{\perp}}$ factorises as the product of the normalised Haar measure of $\AC\slash (\CA_{\fsc})^{\circ}$ and the normalised Haar measure of $(\CA_{\fsc})^{\circ} \slash \CA_{\fc}(\G)^{\circ}$. Since $\eta_{\CA_{\fc}(\G)^{\perp}}$ is ergodic (as shown above), it follows that the normalised Haar measure of $\AC\slash (\CA_{\fsc})^{\circ}$ is ergodic (with respect to the $\G$-action), and this completes the proof (because $\AC \slash (\CA_{\fsc})^{\circ}$ is homeomorphic to $(\CA_{\fsc})^{\perp}$.
\end{proof}

We next describe the \textit{super-Plancherel measure}. For simplicity, we denote by $\lambda_{\fsc}$ the regular character of $G_{\fsc}$; it is clear that $\lambda_{\fsc} \in \SCh_{\G}(G_{\fsc})$ is associated with the normalised Haar measure $\eta_{\fsc}$ on $(\CA_{\fsc})^{\circ}$. Since $\SCh_{\G}(G_{\fsc})$ is a Choquet simplex and $\ISCh_{\G}(G_{\fsc})$ is its set of extreme elements, Choquet's theorem asserts that there is a unique probability measure $\nu_{\fsc}$ on $\ISCh_{\G}(G_{\fsc})$ such that $$\lambda_{\fsc}(g)=\int_{\ISCh_{\G}(\G_{\fsc})}\xi(g)\; d\nu_{\fsc}, \qquad g \in G.$$

\begin{theorem}\label{SuperPlancherel}
Let $G=1+\CA$ be a countable discrete algebra group associated with a nil $\k$-algebra $\CA$, and let $G_{\fsc}=1+\CA_{\fsc}$ be the algebra subgroup consisting of all elements having a finite superclass. If $\nu_{\fsc}$ is the Choquet measure on $\ISCh_{\G}(G_{\fsc})$ associated with the regular character $\lambda_{\fsc}$, then the super-Plancherel measure $\nu$ on $\ISCh(G)$ is given by the pushforward of $\nu_{\fsc}$ by the map $\ISCh_{\G}(G_{\fsc})\to \ISCh(G)$ defined by the mapping $\xi \mapsto \widetilde{\xi}$. In particular, we have $$\lambda_{G}(g)=\int_{\ISCh_{\G}(G_{\fsc})}\widetilde{\xi}(g)\; d\nu_{\fsc}=\int_{\widetilde{\SCh}^{+}_{\G}(G_{\fsc})}\xi(g)\; d\nu,\qquad g \in G,$$ where we set $\widetilde{\SCh}_{\G}^{+}(G_{\fsc}) = \set{\widetilde{\xi}}{\xi \in \ISCh_{\G}(G_{\fsc})}$ (which, by the previous proposition, is a subset of $\ISCh(G)$, easily seen to be closed).
\end{theorem}

\begin{proof}
Since $\widetilde{\lambda}_{\fsc} = \lambda_{G}$, we have $$\lambda_{G}(g)=\int_{\ISCh_{\G}(G_{\fsc})}\widetilde{\xi}(g)\; d\nu_{\fsc},\qquad g \in G.$$ On the other hand, \reft{FteSch} implies that $\widetilde{\xi} \in \ISCh(G)$ for all $\xi \in \ISCh_{\G}(G_{\fsc})$, and the result follows by the uniqueness of the Choquet measure on $\ISCh(G)$ associated with the regular character $\lam_{G}$.
\end{proof} 

We end this section by exemplifying the super-Plancherel decomposition of the regular character in the rather simple case of the unitriangular group $G = U_{n}(\k)$ over an arbitrary infinite countable discrete field $\k$. In this situation, the subgroup $G_{\fsc} = U_{n}(\k)_{\fsc}$ is easy to describe: $U_{n}(\k)_{\fsc} = 1+\k e_{1,n}$ where, for all $1\leq i < j \leq n$, $e_{i,j} \in \fru_{n}(\k)$ denotes the usual elementary matrix having $(i,j)$-th coefficient equal to one and zeroes elsewhere. Notice that $U_{n}(\k)_{\fsc}$ is precisely the center of $U_{n}(\k)$, and that it equals the subgroup $U_{n}(\k)_{\fc}$ consisting of all elements having finite conjugacy class; consequently, the Plancherel measure coincides with the super-Plancherel measure. Since $U_{n}(\k)_{\fsc}$ is clearly isomorphic to the additive group $\k^{+}$ of $\k$, the Pontryagin dual of $U_{n}(\k)_{\fsc}$ may be identified with the Pontryagin dual $\k^{\circ}$ of $\k^{+}$: with every $\tet \in \k^{\circ}$ we associate the character $\map{\xi_{\tet}}{U_{n}(\k)_{\fsc}}{\C}$ defined by $\xi_{\tet}(1+\alpha e_{1,n}) = \tet(\alpha)$ for all $\alpha \in \k$. Therefore, $\ISCh_{\G}(U_{n}(\k)_{\fsc})=\set{\xi_{\tet}}{\tet \in \k^{\circ}}$ where $\G = U_{n}(\k) \x U_{n}(\k)$, and thus $$\lambda_{U_{n}(\k)}(g) = \int_{\k^{\circ}}\widetilde{\xi_{\tet}}(g) \; d\nu$$ where $\nu$ is the normalised Haar measure on $\k^{\circ}$ and where, for every $\tet \in \k^{\circ}$, $\widetilde{\xi_{\tet}}$ denotes the extension by zero of $\xi_{\tet}$ from $U_{n}(\k)_{\fsc}$ to $U_{n}(\k)$.

The behaviour of the regular character of $U_{n}(\k)$ provides a good example of how different phenomena arises in the representation of \textit{big groups}: although the regular character of any finite algebra group decomposes as a convex sum of \textit{all} its indecomposable supercharacters, this is no longer true in the case of the infinite unitriangular group $U_n(\k)$, where the regular character decomposes as a ``convex sum'' involving only a ``very small'' set of indecomposable supercharacters.

\section{The infinite unitriangular groups $U_{n}(\k)$} \label{unitr}

This section is mainly devoted to description of the supercharacters of the infinite unitriangular group $U_{n}(\k)$ where $\k$ is the algebraic closure of a finite field of characteristic $p$. The group $U_{n}(\k)$ is the prototype example of an \textit{approximately finite algebra group}: an algebra group $G=1+\CA$ over a field $\k$ (associated with a nil $\k$-algebra $\CA$) is said to be \textit{approximately finite} if there are a chain $\k_{1} \sset \k_{2} \sset \cdots \sset \k_{m} \sset \cdots$ of subfields of $\k$ and a chain $G_{1} \sset G_{2} \sset \cdots \sset G_{m} \sset \cdots$ of finite subgroups of $G$ such that, for every $m \in \N$, the subgroup $G_{m} = 1+\CA_{m}$ is an algebra group over $\k_{m}$ associated with some $\k_{m}$-subalgebra $\CA_{m}$ of $\CA$, and $$G = \bigcup_{m\in\N} G_{m};$$ for simplicity, we refer to such an algebra group $G$ as an \textit{AF-algebra group}. We note that, since $G_{m}$ is a finite group, the subfield $\k_{m}$ must be finite, and hence $\k$ must have non-zero characteristic (as required); moreover, the $\k_{m}$-algebra $\CA_{m}$ must be finite-dimensional over $\k_{m}$, and hence it is a nilpotent $\k_{m}$-subalgebra of $\CA$. On the other hand, the inclusion $G_{m} \sset G_{m+1}$ clearly implies that $\CA_{m} \sset \CA_{m+1}$, and thus $\CA_{m}$ is a $\k_{m}$-subalgebra of $\CA_{m+1}$. Finally, we observe that, with respect to the natural inclusion maps, the AF-algebra group $G$ is isomorphic to the direct limit $\varinjlim_{m \in \N} G_{m}$, and hence it is an amenable group (by \cite[Proposition~13.6]{Pier1984a} because finite groups are trivially amenable); furthermore, it is also straightforward to check that $\k = \bigcup_{m\in \N} \k_{m}$ is also isomorphic to the direct limit $\varinjlim_{m \in \N} \k_{m}$.

Concerning our standard example, for every $m \in \N$, let $\k_{m}$ denote the finite field $\F_{p^{m!}}$ with $p^{m!}$ elements where $p$ is the characteristic of $\k$. Since $\k_{m}$ is a subfield of $\k_{m+1}$, we may consider the (finite) unitriangular group $U_{n}(\k_{m})$ as a subgroup of $U_{n}(\k_{m+1})$. In this situation, $\k = \bigcup_{m\in \N} \k_{m}$ is the algebraic closure of the prime field $\F_{p}$ with $p$ elements, and the (infinite) unitriangular group $U_{n}(\k)$ may be naturally realised as the union $$U_{n}(\k) = \bigcup_{m\in \N} U_{n}(\k_{m});$$ in particular, we see that $U_{n}(\k)$ is an AF-algebra group.

In the general situation, let $G = \bigcup_{m \in \N} G_{m}$ be an arbitrary (fixed) AF-algebra group, and consider the set $\SCl(G)$ of superclasses and the set $\SCh(G))$ of supercharacters of $G$ (notice that $G$ is countable and discrete); moreover, for every $m \in \N$, we consider  the set $\SCl(G_{m})$ of superclasses and the set $\SCh(G_{m})$ of supercharacters of the finite algebra group $G_{m}=1+\CA_{m}$. Our aim is to exploit the relationship between the pair $(\SCl(G),\ISCh(G))$ and the pairs $(\SCl(G_{m}),\ISCh(G_{m}))$ for every $m \in \N$, in order to show that every indecomposable supercharacter $\xi \in \ISCh(G)$ is \textit{finitely approximated} by a sequence $(\xi_{m})_{m \in \N}$, where $\xi_{m} \in \ISCh(G_{m})$ for all $m \in \N$, in the sense that $\xi$ is the pointwise limit of $(\xi_{m})_{m\in \N}$. This will be accomplished using Lindenstrauss' pointwise ergodic theorem for \textit{tempered F\o lner sequences}; we recall that a \textit{F\o lner sequence} for a group $G$ is a family $\set{F_{m}}{m \in \N}$ of finite subsets of $G$ such that, for every $g \in G$, $$\lim_{m\to\infty} \frac{|F_{m} \bigtriangleup (gF_{m})|}{|F_{m}|} = 0$$ where $\bigtriangleup$ denotes the symmetric difference of sets, and that a F\o lner sequence $\set{F_{m}}{m \in \N}$ for an amenable group $G$ is said to be \textit{tempered} if there is a positive real number $C > 0$ such that $$\bigg| \bigcup_{1 \leq k \leq m} F_{k}^{-1} F_{m+1} \bigg| \leq C\,\big| F_{m+1} \big|,\qquad m \in \N.$$

\begin{Lindtheorem}[\mbox{\cite[Theorem~1.3]{Lindenstrauss2001a}}]
Let $\CG$ be an amenable discrete group acting on a probability space $(X,\mu)$. Assume that the measure $\mu$ is ergodic with respect to the $\CG$-action on $X$, and that there is a tempered F\o lner sequence $\{F_{m} \}_{m \in \N}$ for $\CG$. Then, for $\mu$-almost every point $x \in X$ and every $f \in L^{1}(X,\mu)$, we have $$\lim_{m\to\infty} \frac{1}{|F_{m}|} \sum_{g \in F_{m}} f(g\inv\cdot x) = \int_{X} f(x)\; d\mu.$$
\end{Lindtheorem}

For every $m \in \N$, let $\G_{m} = G_{m} \x G_{m}$, and note that $\G = \bigcup_{m\in\N} \G_{m}$; as before, we set $\G = G\x G$. For every $\g \in \G$, there is $m \in \N$ such that $\g \in \G_{m}$, and thus the symmetric difference $\G_{m'} \bigtriangleup (\g\G_{m'})$ is empty for all $m' \in \N$ such that $m' \geq m$. Therefore, $\set{\G_{m}}{m \in \N}$ is obviously a F\o lner sequence for $\G$ which is clearly tempered. As a consequence of Lindenstrauss' pointwise ergodic theorem, we deduce the following result; we recall that, for every $m \in \N$, the indecomposable supercharacters of the finite algebra group $G_{m} = 1+\CA_{m}$ are in one-to-one correspondence with the $\G_{m}$-orbits on the dual group $(\CA_{m})^{\circ}$ (by means of the formula \eqref{FSCh}).

\begin{theorem}[Finite approximation property] \label{FiniteApprox}
In the notation as above, let $G = \bigcup_{m \in \N} G_{m}$ be an arbitrary AF-algebra group, let $\m \in \CM^{+}_{\G}(\AC)$ be an ergodic $\G$-invariant measure on $\AC$, and consider the indecomposable supercharacter $\xi^{\mu} \in \ISCh(G)$ which (uniquely) corresponds to $\mu$; hence, $$\xi^{\mu}(g) = \int_{\AC} \tet(g-1)\;d\m, \qquad g \in G.$$ For every $\tet \in \AC$, let $\tet_{m} \in (\CA_{m})^{\circ}$ denote the restriction of $\tet$ to $\CA_{m}$, and let $\xi_{m} \in \ISCh(G_{m})$ be the indecomposable supercharacter of $G_{m}$ which corresponds to the $\G_{m}$-orbit $\G_{m}\cdot \tet_{m}$. Then, for $\m$-almost every $\tet \in \AC$, we have $$\xi^{\mu}(g) = \lim_{m \to \infty} \xi_{m}(g), \qquad g \in G.$$ Furthermore, every indecomposable supercharacter of $G$ is of this form.
\end{theorem}

\begin{proof}
For every $m \in \N$, every $\tet \in \AC$ and every $g \in G$, we set $$A_{m,\tet}(g) = \frac{1}{|\G_{m}|} \sum_{\g \in \G_{m}} \ev_{g}(\g \cdot \tet) = \frac{1}{|\G_{m}|} \sum_{\g \in \G_{m}} (\g \cdot \tet)(g-1);$$ we note that, since $(\g\cdot\tet)_{m} = \g\cdot \tet_{m}$  for all $\g \in \G_{m}$, we have $A_{m,\tet}(g) = \xi_{m}(g)$ whenever $g \in G_{m}$ (see \refeq{FSCh}).

On the one hand, let $\tet \in \AC$ be fixed, and assume that the sequence $(A_{m,\tet}(g))_{m\in \N}$ converges for all $g \in G$. Then, as in the proof of \refp{AmenableOrbital}, we may define an ergodic $\G$-invariant measure $\mu$ on $\AC$ by the rule $$\int_{\AC} f \;d\mu = \lim_{m \to \infty} \frac{1}{|\G_{m}|} \sum_{\g \in \G_{m}} f(\g\cdot\tet),\qquad f \in C(\AC);$$ in particular, we see that $$\xi^{\mu}(g) = \int_{\AC} \ev_{g} \;d\mu = \lim_{m \to \infty} A_{m,\tet}(g),\qquad g \in G.$$ Now, let $g \in G$ be arbitrary, and choose the smallest $m_{0} \in \N$ such that $g \in G_{m_{0}}$. Then, $A_{m,\tet}(g) = \xi_{m}(g)$ for all $m \in \N$ such that $m \geq m_{0}$ (because $g \in G_{m}$), and thus $\xi^{\mu}(g) = \lim_{m \to \infty}\xi_{m}(g)$, as required.

Conversely, if $\m$ is an arbitrary ergodic $\G$-invariant measure on $\AC$, then Lindenstrauss' pointwise ergodic theorem implies that, for $\m$-almost every $\tet \in \AC$, we have  $$\xi^{\m}(g) = \int_{\AC} \ev_{g}\;d\mu = \lim_{m \to \infty}\frac{1}{|\G_{m}|} \sum_{\g \in \G_{m}} \ev_{g}(\g \cdot \tet) = \lim_{m \to \infty} A_{m,\tet}(g), \qquad g \in G$$ (recall that $\set{\G_{m}}{m \in \N}$ is a tempered F\o lner sequence). Since $A_{m,\tet}(g) = \xi_{m}(g)$ for all sufficiently large $m \in \N$, the result is now clear.
\end{proof}

The previous theorem exhibits the \textit{asymptotic nature} of the indecomposable supercharacters of AF-algebra groups; in the case of indecomposable characters of an arbitrary AF-group (that is, a direct limit of finite groups) an analogous result was established by Kerov and Vershik  (see \cite{Vershik1981a}). In the particular example of the unitriangular group $U_{n}(\k)$ over the algebraic closure of a finite field of non-zero characteristic $p$, since there are explicit formulas for the indecomposable supercharacters of the finite unitriangular groups, it is possible to derive a formula for its indecomposable supercharacters; in \cite{Andre2018a}, a similar formula was accomplished (albeit by different methods, based on \cite{Vershik1981a}) in the case of the locally finite unitriangular group $U_{\infty}(\fq) = \bigcup_{n\in\N} U_{n}(\fq)$ over an arbitrary finite field $\fq$.

We shall describe the supercharacters of $U_{n}(\k)$ using the finite approximation property (\reft{FiniteApprox}); for simplicity, we set $\CK = \SCl(U_{n}(\k))$ and $\CE = \ISCh(U_{n}(\k))$. We start by providing a brief characterisation of the superclasses and indecomposable supercharacters of a finite unitriangular group; the details can be found in \cite{Andre1995a,Andre2002a} (see also \cite{Andre2013a}). For every $m \in \N$, we set $\CK_{m} = \SCl(U_{n}(\k_{m}))$ and $\CE_{m} = \ISCh(U_{n}(\k_{m}))$; recall that $\k_{m}$ denotes the finite field $\F_{p^{m!}}$ with $p^{m!}$ elements (which we consider as a subfield of $\k$).

We denote by $\SP(n)$ the set consisting of all set partitions of $[n]=\{1, \ldots, n\}$, and write $\pi \in \SP(n)$ as a sequence $\pi = B_{1}\slash B_{2}\slash \ldots \slash B_{k}$ where $\seq{B}{k}$ are disjoint subsets of $[n]$ such that $[n] = B_{1} \cup B_{2} \cup \cdots \cup B_{k}$; we refer to $\seq{B}{k}$ as the \textit{blocks} of $\pi$. A pair $(i,j)$ with $1 \leq i < j \leq n$ is said to be an \textit{arc} of $\pi \in \SP(n)$ if $i$ and $j$ lie in the same block $B$ of $\pi$ and there is no $k$ in $B$ such that $i<k<j$; we denote by $\CD(\pi)$ the set consisting of all arcs of $\pi$. If $\pi \in \SP(n)$, then a map $\map{\alpha}{\CD(\pi)}{\k_{m}\setminus\{0\}}$ is called a $\k_{m}$-\textit{colouration} of $\pi$. We denote by $\Col_{\k_{m}}(\pi)$ the set consisting of all $\k_{m}$-colourations of $\pi \in \SP(n)$, and define $$\Phi_{n}(\k_{m}) = \set{(\pi,\alpha)}{\pi \in \SP(n),\ \alpha \in \Col_{\k_{m}}(\pi)};$$ an element of $\Phi_{n}(\k_{m})$ is referred to as a \textit{$\k_{m}$-coloured set partition} of $[n]$.

For every $(\pi,\alpha) \in \Phi_{n}(\k_{m})$, we define $e_{\pi,\alpha} \in \fru_{n}(\k_{m})$ to be the element $$e_{\pi,\alpha} = \sum_{(i,j)\in \CD(\sigma)} \alpha(i,j) e_{i,j}$$ where $e_{i,j} \in \fru_{n}(\k_{m})$ stands for the elementary matrix having $(i,j)$-th coefficient equal to $1$ and zeroes elsewhere. It is straightforward to check that, for every superclass $K \in \CK_{m}$, there is a unique $(\pi,\alpha) \in \Phi_{n}(\k_{m})$ such that $1+e_{\pi,\alpha} \in K$, and thus the superclasses in $\CK_{m}$ are in one-to-one correspondence with the $\k_{m}$-coloured set partitions of $[n]$; we denote by $K_{\pi,\alpha}$ the superclass in $\CK_{m}$ which is associated with $(\pi,\alpha) \in \Phi_{n}(\k_{m})$.

We next proceed with the description of the indecomposable supercharacters of $U_{n}(\k_{m})$. Let $\kc_{m}$ denote the dual group of the additive group $\k_{m}^{+}$; notice that there is a group isomorphism $\k^{+}_{m} \cong \kc_{m}$ given by the mapping $\alpha \mapsto \alpha\tau$ where $\tau \in \kc_{m}$ is an arbitrarily fixed non-trivial character of $\k_{m}$, and where the character $\alpha\tau \in \kc_{m}$, for $\alpha \in \k_{m}$, is defined $(\alpha\tau)(\bet) = \tau(\alpha\bet)$ for all $\bet \in \k_{m}$. Since the additive group $\fru_{n}(\k_{m})^{+}$ is clearly isomorphic to a direct product of $n(n-1)\slash 2$ copies of $\k^{+}_{m}$, every character $\tet \in \fru_{n}(\k_{m})^{\circ}$ can be uniquely represented by a strictly uppertriangular $n\times n$ matrix $\tet = (\tet_{i,j})$ where $\tet_{i,j} \in \kc_{m}$ for all $1 \leq i < j \leq n$; in fact, if $\tau \in \kc_{m}$ is an arbitrarily fixed non-trivial character of $\k_{m}$ and if we define $\tau_{a}(b) = \tau(\tr(a\trp b))$ for all $a,b \in \fru_{n}(\k_{m})$, where $\tr$ denotes the usual trace map and $a\trp$ denotes the transpose matrix of $a$, then it is not hard to check that the mapping $a \mapsto \tau_{a}$ defines a group isomorphism between $\fru_{n}(\k_{m})^{+}$ and $\fru_{n}(\k_{m})^{\circ}$ (we observe that the existence of this isomorphism depends strongly on the self-duality of the finite field $\k_{m}$, and that a similar isomorphism cannot exist when we replace $\k_{m}$ by its algebraic closure $\k$). For our purposes, it is useful to extend this notation and define the character $\tau_{a} \in \fru_{n}(\k_{m})^{\circ}$ for all $\tau \in \kc_{m}$ and all $a \in \fru_{n}(\k_{m})$; for simplicity of writing, we set $\tau\star a^{\ast} = \tau_{a}$. By the way of example, for every $\tau \in \kc_{m}$ and every $1 \leq i < j \leq n$, the character $\tau \star e_{i,j}^{\ast} \in \fru_{n}(\k_{m})^{\circ}$ is given by $(\tau \star e_{i,j}^{\ast})(b) = \tau(b_{i,j})$ for all $b = (b_{i,j}) \in \fru_{n}(\k_{m})$, and thus we see that every character $\tet \in \fru_{n}(\k_{m})^{\circ}$ decomposes uniquely as the sum $$\tet = \sum_{1 \leq i < j \leq n} \tet_{i,j} \star e_{i,j}^{\ast}$$ where we are using the additive notation in $\fru_{n}(\k_{m})^{\circ}$ and where $\tet_{i,j} \in \kc_{m}$, for $1 \leq i < j \leq n$, is defined by $\tet_{i,j}(\alpha) = \tet(\alpha e_{i,j})$ for all $\alpha \in \k_{m}$; hence, $\tet$ is represented by the strictly uppertriangular matrix $(\tet_{i,j})$.

Following the previous terminology, by a $\kc_{m}$-\textit{colouration} of $\pi \in \SP(n)$ we understand a map $\map{\tau}{\CD(\pi)}{\kc_{m}\setminus\{\0\}}$ where $\0 \in \kc_{m}$ stands for the trivial character of $\k^{+}_{m}$. We denote by $\Col_{\kc_{m}}(\pi)$ the set consisting of all $\kc_{m}$-colourations of $\pi \in \SP(n)$, and define $$\Phi_{n}(\kc_{m}) = \set{(\pi,\tau)}{\pi \in \SP(n),\ \tau \in \Col_{\kc_{m}}(\pi)};$$ an element of $\Phi_{n}(\kc_{m})$ is referred to as a \textit{$\kc_{m}$-coloured set partition} of $[n]$. For every $(\pi,\tau) \in \Phi_{n}(\kc_{m})$, we define $\tet_{\pi,\tau} \in \fru_{n}(\k_{m})^{\circ}$ to be the character $$\tet_{\pi,\tau} = \sum_{(i,j)\in \CD(\pi)} \tau(i,j) \star e^{\ast}_{i,j};$$ hence, $\tet_{\pi,\tet}$ is represented by the matrix $(\tet_{i,j})$ where, for all $1 \leq i < j \leq n$, we set $\tet_{i,j} = \tau(i,j)$ if $(i,j) \in \CD(\pi)$, and $\tet_{i,j} = \0$, otherwise. As before, it is routine to check that, for every $\tet \in \fru_{n}(\k_{m})^{\circ}$, there is a unique $(\pi,\tau) \in \Phi_{n}(\kc_{m})$ such that $\tet_{\pi,\tau} \in \G_{m}\cdot\tet$, and thus the indecomposable supercharacters in $\CE_{m}$ are in one-to-one correspondence with the $\kc_{m}$-coloured set partitions of $[n]$; we denote by $\xi^{\pi,\tau}$ the indecomposable supercharacter in $\CE_{m}$ which is associated with $(\pi,\tau) \in \Phi_{n}(\kc_{m})$.

Finally, for every $\pi \in \SP(n)$, we define the sets $$\CS(\pi) = \{(i,l), (k,j) \colon (i,j) \in \CD(\pi),\ j < l \leq n,\ 1 \leq k < i\},\quad \text{and}$$ $$\CR(\pi)=\{(i,j) \colon 1\leq i < j \leq n \} \setminus \CS(\pi);$$ moreover, for every $1 \leq i < j\leq n$ and every $\pi,\pi' \in \SP(n)$, we define the \textit{nesting numbers} $$\nest_{(i,j)}(\pi') = \big|\set{(k,l) \in \CD(\pi')}{i < k < l < j}\big|,\quad \text{and}$$ $$\nest_{\pi}(\pi') = \sum_{(i,j)\in \CD(\pi)} \nest_{(i,j)}(\pi').$$ The following formula for the (normalised) indecomposable supercharacter values can be easily derived from \cite[Theorem~3]{Andre2002a} (see also \cite[Theorem~5.1]{Andre2001a}). 

\begin{proposition} \label{SupercharacterFormula}
Let $(\pi,\tau) \in \Phi_{n}(\kc_{m})$ and $(\pi',\alpha) \in \Phi_{n}(\k_{m})$ be arbitrary. Then, the (constant) value $\xi^{\pi,\tau}(K_{\pi',\alpha})$ of the (normalised) indecomposable supercharacter $\xi^{\pi,\tau}$ on the superclass $K_{\pi',\alpha}$ is equal to $0$ unless $\CD(\pi) \subseteq \CR(\pi')$, in which case it is given by $$\xi^{\pi,\tau}(K_{\pi',\alpha}) = \frac{1}{p^{m!\nest_{\pi}(\pi')}}\; \tet_{\pi,\tau}(e_{\pi',\alpha}).$$
\end{proposition}

We are now able to deal with the infinite unitriangular group $U_{n}(\k)$, and to describe its superclasses and its indecomposable supercharacters. As above, we denote by $\Phi_{n}(\k)$ the set consisting of all $\k$-coloured set partitions of $[n]$; by a \textit{$\k$-colouration} of a set partition $\pi \in \SP(n)$ we mean a map $\map{\alpha}{\CD(\pi)}{\k \setminus\{0\}}$. Notice that there is an obvious inclusion $\Phi_{n}(\k_{m}) \subseteq \Phi_{n}(\k_{m+1})$ for all $m \in \N$, and that $$\Phi_{n}(\k) = \bigcup_{m\in \N} \Phi_{n}(\k_{m}).$$ On the other hand, for every $m \in \N$, let $K^{(m)}_{\pi,\alpha} \in \CK_{m}$ denote the superclass of $U_{n}(\k_{m})$ which is parametrised by the set partition  $(\pi,\alpha) \in \Phi_{n}(\k_{m})$; hence, $K^{(m)}_{\pi,\alpha} = 1+\G_{m}\cdot e_{\pi,\alpha}$. Since $\G_{m} \sset \G_{m+1}$, it is clear that $K^{(m)}_{\pi,\alpha} \sset K^{(m+1)}_{\pi,\alpha}$ for all $(\pi,\alpha) \in \Phi_{n}(\k_{m})$ and all $m\in\N$; consequently, since $\fru_{n}(\k) = \bigcup_{m\in \N} \fru_{n}(\k_{m})$, for every superclass $K$ of $U_{n}(k)$, there is a unique $(\pi,\alpha) \in \Phi_{n}(\k)$ such that $1+e_{\pi,\alpha} \in K$. Therefore, as in the case of the finite unitriangular groups, we conclude that the superclasses of $U_{n}(\k)$ are in one-to-one correspondence with the $\k$-coloured set partitions; as before, we denote by $K_{\pi,\alpha}$ the superclass in $\CK$ which is associated with $(\pi,\alpha) \in \Phi_{n}(\k)$.

For the characterisation of $\CE = \ISCh(U_{n}(\k))$, we first note that (as in the finite case) every character $\tet \in \fru_{n}(\k)^{\circ}$ can be uniquely represented by a strictly uppertriangular $n\x n$ matrix $\tet = (\tet_{i,j})$ where $\tet_{i,j} \in \kc$ for all $1 \leq i < j \leq n$; indeed, $\tet$ decomposes uniquely as the sum $$\tet = \sum_{1 \leq i < j \leq n} \tet_{i,j} \star e_{i,j}^{\ast}$$ where we define $(\tau\star a^{\star})(b) = \tau(\tr(a\trp b)$ for all $\tau \in \kc$ and all $a,b \in \fru_{n}(\k)$. On the other hand, since
$\fru_{n}(\k) = \bigcup_{m\in\N} \fru_{n}(\k_{m})$ is isomorphic to the direct limit $\varinjlim_{m \in \N} \fru_{n}(\k_{m})$, there is an isomorphism of abelian groups $$\fru_{n}(\k)^{\circ} = \Hom(\fru_{n}(\k),\C^{\x}) \cong \varprojlim_{m\in\N} \Hom(\fru_{n}(\k_{m}),\C^{\x}) = \varprojlim_{m\in\N} \fru_{n}(\k_{m})^{\circ}$$ where, for every $m \in \N$, the restriction map $\fru_{n}(\k_{m+1})^{\circ} \to \fru_{n}(\k_{m})^{\circ}$ is naturally given by the restriction of characters; in particular, every character $\tet \in \fru_{n}(\k)^{\circ}$ may be viewed as the infinite sequence $(\tet_{m})_{m\in \N}$ where, for every $m \in \N$, we define $\tet_{m} \in \fru_{n}(\k_{m})^{\circ}$ to be the restriction of $\tet$ to $\fru_{n}(\k_{m})$. Similarly, since $\k \cong \varinjlim_{m\in \N} \k_{m}$, there is an isomorphism of abelian groups $$\kc \cong \varprojlim_{m\in\N} \kc_{m},$$ and thus every character $\sigma \in \kc$ may be viewed as the infinite sequence $(\sigma_{m})_{m\in \N}$ where, for every $m \in \N$, $\sigma_{m} \in \kc_{m}$ denotes the restriction of $\sigma$ to $\k_{m}$. (We observe that, since there is an obvious group homomorphism $\k^{+} \cong U_{2}(\k)$, \reft{FiniteApprox} implies that every character $\sigma \in \kc$ is the pointwise limit of the sequence $(\sigma_{m})_{m\in \N}$; in other words, we have $\sigma(\alpha) = \lim_{m\to\infty} \sigma_{m}(\alpha)$ for all $\alpha \in \k$.)

Now, let $\Phi_{n}(\kc)$ stand for the set of all \textit{$\kc$-coloured set partitions} of $[n]$; hence, an element of $\Phi_{n}(\kc)$ is a pair $(\pi,\tau)$ where $\pi \in \SP(n)$ and $\map{\tau}{\CD(\pi)}{\kc\setminus \{\0\}}$ is a map (to which we refer as a \textit{$\kc$-colouration} of $\pi$). For every $(\pi,\tau) \in \Phi_{n}(\kc)$ and every $m \in \N$, we define $\map{\tau_{m}}{\CD(\pi)}{\kc_{m} \setminus \{ \0 \}}$ by $\tau_{m}(i,j) = \tau(i,j)_{m}$ for all $(i,j) \in \CD(\pi)$; we notice that, with respect to the obvious restriction maps, the set $\Phi_{n}(\kc)$ can be realised as the inverse limit $$\Phi_{n}(\kc) = \varprojlim_{m \in \N}\Phi_n(\kc_{m})$$ where a $(\pi,\tau) \in \Phi_{n}(\kc)$ is viewed as the infinite sequence $((\pi,\tau_{m}))_{m\in \N}$. Now, for every $(\pi,\tau) \in \Phi_{n}(\kc)$, we define the character $\tet_{\pi,\tau} \in \fru_{n}(\k)^{\circ}$ by $$\tet_{\pi,\tau} = \sum_{(i,j) \in \CD(\pi)} \tau(i,j) \star e^{\ast}_{i,j};$$ 
on the other hand, if the sequence $(\xi^{\pi,\tau_{m}}(g))_{m \in \N}$ converges for all $g \in G$, then we define the function $\map{\xi^{\pi,\tau}}{G}{\C}$ by $$\xi^{\pi,\tau}(g) = \lim_{m \to \infty}\xi^{\pi,\tau_{m}}(g), \qquad g \in G,$$ and observe that, as a consequence of \reft{FiniteApprox}, the function $\xi^{\pi,\tau}$, whenever defined, is an indecomposable supercharacter of $U_{n}(\k)$. Indeed, the following result holds.

\begin{proposition} \label{SChUnit}
The indecomposable supercharacter $\xi^{\pi,\tau}$ of $U_{n}(\k)$ is defined for all $(\pi,\tau) \in \Phi_{n}(\kc)$. Moreover, for every $(\pi,\tau) \in \Phi_{n}(\kc)$ and every $(\pi',\alpha) \in \Phi_{n}(\k)$, the constant value $\xi^{\pi',\tau}(K_{\pi',\alpha})$ of $\xi^{\pi,\tau}$ on the superclass $K_{\pi',\alpha}$ is equal to $0$ unless $\CD(\pi) \subseteq \CR(\pi')$ and $\nest_{\pi}(\pi') = 0$, in which case it is given by $\xi^{\pi,\tau}(K_{\pi',\alpha}) = \tet_{\pi,\tau}(e_{\pi',\alpha})$.
\end{proposition}

\begin{proof}
Firstly, we note that, if $\sigma \in \kc$ is non-trivial, then there is $m'_{0} \in \N$ such that $\sigma_{m} \neq \0$ for all $m \in \N$ such that $m \geq m'_{0}$, and consequently, for an arbitrary $(\pi,\tau) \in \Phi_{n}(\kc)$, there is $m_{0} \in \N$ such that $\tau(i,j)_{m} \neq \0$ for all $m \in \N$ such that $m \geq m_{0}$ and all $(i,j) \in \CD(\pi)$. Therefore, for every $m \in \N$ such that $m \geq m_{0}$, the mapping $(i,j) \mapsto \tau(i,j)_{m}$ defines the $\kc_{m}$-colouration $\map{\tau_{m}}{\CD(\pi)}{\kc_{m}\setminus \{\0\}}$, and hence $(\pi,\tau_{m}) \in \Phi_{n}(\kc_{m})$; moreover, we have $$\lim_{m\to\infty} \tau(i,j)_{m}(\bet) = \tau(i,j)(\bet), \qquad \bet \in \k.$$ On the other hand, \refp{SupercharacterFormula} asserts that, for every $m \in \N$ such that $m \geq m_{0}$, the value $\xi^{\pi,\tau_{m}}(K_{\pi',\alpha})$ is non-zero only if $\CD(\pi) \sset \CR(\pi')$, in which case we have $$\xi^{\pi,\tau_{m}}(K_{\pi',\alpha}) = \frac{1}{p^{m!\nest_{\pi}(\pi')}}\; \tet_{\pi,\tau_{m}}(e_{\pi',\alpha}).$$ Therefore, $\xi^{\pi,\tau}(g) = \lim_{m \to \infty}\xi^{\pi,\tau_{m}}$ exists for all $g \in G$, and in the case where $\CD(\pi) \sset \CR(\pi')$, we conclude that $$\xi^{\pi,\tau}(K_{\pi',\alpha}) = \begin{cases} 0, & \text{if $\nest_{\pi}(\pi') \neq 0$,} \\ \tet_{\pi,\tau}(e_{\pi',\alpha}), & \text{if $\nest_{\pi}(\pi') = 0$,} \end{cases}$$ and this concludes the proof.
\end{proof}

As a consequence of \reft{FiniteApprox}, we obtain the following result. 

\begin{theorem}
The mapping $(\pi,\tau) \mapsto \xi^{\pi,\tau}$ defines a one-to-one correspondence between $\Phi_{n}(\kc)$ and $\ISCh(U_{n}(\k))$. Furthermore, this correspondence defines a homeomorphism when $\Phi_{n}(\kc)=\varprojlim_{m \in \N}\Phi_{n}(\kc_{m})$ is equipped with the inverse limit topology (where, for every $m \in \N$, the space $\Phi_{n}(\kc_{m})$ is equipped with the discrete topology).
\end{theorem}

The supercharacters of $U_{n}(\k)$ enjoy (very) particular features which generalise important properties of the supercharacters of the finite unitriangular groups (and, more generally, of finite algebra groups). Namely, the closure of every $\G$-orbit on $\fru_{n}(\k)$ supports a single ergodic $\G$-invariant measure; if $\CO$ is the closure of some $\G$-orbit, then every ergodic $\G$-invariant measure supported on $\CO$ must determine an indecomposable supercharacter, and thus $\CO$ contains one element of the form $\tet_{\pi,\tau}$ for some $(\pi,\tau) \in \Phi_{n}(\kc)$. However, one can check that, if $(\pi,\tau)$ and $(\pi',\tau')$ are two distinct elements of $\Phi_n(\kc)$, then the closures of the $\G$-orbits $\G \cdot \vartheta_{\pi,\tau}$ and $\G \cdot \vartheta_{\pi',\tau'}$ are also distinct, and consequently the closure of every $\G$-orbit supports a unique  ergodic $\G$-invariant measure. 

In more detail, let $1\leq i < j \leq n$ be arbitrary, and let $\pi \in \SP(n)$ be the unique set partition such that $\CD(\pi) = \{(i,j)\}$. For every $\tau \in \kc \setminus \{\0\}$, let $\xi_{i,j}(\tau)$ denote the supercharacter $\xi^{\pi,\tau}$ where $\tau \in \Col_{\kc}(\pi)$ is given by $\tau(i,j) = \tau$; following \cite{Andre1995a}, we refer to $\xi_{i,j}(\tau)$ as the $(i,j)$-th \textit{elementary character} of $U_{n}(\k)$ associated with $\tau$. On the other hand, let $\tet_{i,j}(\tau) = \tau \star e^{\ast}_{i,j} \in \fru_{n}(\k)^{\circ}$, and let $\CO_{i,j}(\tau)$ denote the closure of the $\G$-orbit $\G\cdot \tet_{i,j}(\tau)$. It is straightforward to check that $\G\cdot \tet_{i,j}(\tau)$ consists of all characters $\tau \star a^{\ast} \in \fru_{n}(\k)^{\circ}$ where $a = (a_{r,s}) \in \fru_{n}(\k)$ is any element which satisfies $$a_{r,s} = \begin{cases} 0, & \text{if $1 \leq r < i$ or $j < s \leq n$,} \\ 1, & \text{if $r = i$ and $s = j$,} \\ a_{i,s}a_{r,j}, & \text{if $i < r < s < j$,} \end{cases}$$ and thus the $\G$-orbit $\G\cdot \tet_{i,j}(\tau)$ is homeomorphic to the product space $\k^{2(j-i-1)}$ (hence, it is countable and discrete). On the other hand, for every $\tau \in \kc$, the set $S = \set{\alpha\tau}{\alpha \in \k}$ is clearly a subgroup of $\kc$ with $S^{\perp} = \{0\}$ where $S^{\perp}$ is the orthogonal subgroup of $S$ in $\k$; since $\ovl{S}{}^{\perp} = S^{\perp}$ where $\ovl{S}$ denotes the closure of $S$ in $\kc$, we conclude that $\ovl{S} = \ovl{S}{}^{\perp\perp} = \kc$ (hence, $S$ is dense in $\kc$), and this implies that the closure $\CO_{i,j}(\tau)$ of $\G\cdot\tet_{i,j}(\tau)$ is homeomorphic to the compact space $(\kc)^{2(j-i-1)}$.

In general, it can be proved that, for every $(\pi,\tau) \in \Phi_{n}(\kc)$, the $\G$-orbit $\G \cdot \tet_{\pi,\tau}$ decomposes as the (algebraic) sum $$\G \cdot \tet_{\pi,\tau} = \sum_{(i,j)\in \CD(\pi)} \G\cdot \tet_{i,j}(\tau(i,j)),$$ and consequently its closure $\CO^{\pi,\tau}$ decomposes as the sum $$\CO^{\pi,\tau} = \sum_{(i,j)\in \CD(\pi)} \CO_{i,j}(\tau(i,j)).$$ From this, it is not hard to conclude that $\CO^{\pi,\tau}$ is homeomorphic to the compact space $(\kc)^{r(\pi)}$ where $r(\pi) = \big| \bigcup_{(i,j) \in \CD(\pi)} \set{(i,k), (k,j)}{i < k < j}\big|$.

Since there is a one-to-one correspondence between $\ISCh(U_{n}(\k))$ and the set consisting of the closures of all $\G$-orbit on $\fru_{n}(\k)^{\circ}$, it follows that, if $\CO$ is the closure of some $\G$-orbit on $\fru_{n}(\k)^{\circ}$ and $\mu^{\CO}$ is the unique ergodic $\G$-invariant measure supported on $\CO$, then the corresponding indecomposable supercharacter $\xi^{\CO} \in \ISCh(U_{n}(\k))$ admits the integral formula $$\xi^{\CO}(g) = \int_{\CO} \tet(g)\; d\mu^{\CO},\qquad g \in G,$$ which is, not only a generalisation of the supercharacter formula for finite algebra groups, but also an analogue for the Kirillov character formula for nilpotent real Lie groups (see \cite{Kirillov1962a} and also \cite{Corwin1994a} for the case of nilpotent discrete groups over the rational numbers). Furthermore, it is easy to check that, in the particular case of the discrete Heisenberg group $U_{3}(\k)$, every the superclasses are precisely the conjugacy classes, and hence $\ISCh(U_{3}(\k)) = \ICh(U_{3}(\k))$. These similarities, together with the fact that the Plancherel measure coincides with the super-Plancherel measure, suggests that supercharacter theories may provide adequate alternatives to the classical character theory.

\bibliographystyle{acm}

\begin{thebibliography}{10}

\bibitem{Alfsen1971a}
{\sc Alfsen, E.M.}
\newblock {\em Compact convex sets and boundary integrals}.
\newblock Springer-Verlag, New York-Heidelberg, 1971.
\newblock Ergebnisse der Mathematik und ihrer Grenzgebiete, Band 57.

\bibitem{Andre1995a}
{\sc {Andr\'e}, C.A.M.}
\newblock {Basic characters of the unitriangular group.}
\newblock {\em {J. Algebra} 175}, 1 (1995), 287--319.

\bibitem{Andre1995b}
{\sc {Andr\'e}, C.A.M.}
\newblock {Basic sums of coadjoint orbits of the unitriangular group.}
\newblock {\em {J. Algebra} 176}, 3 (1995), 959--1000.

\bibitem{Andre1998a}
{\sc {Andr\'e}, C.A.M.}
\newblock {The regular character of the unitriangular group.}
\newblock {\em {J. Algebra} 201}, 1 (1998), 1--52.

\bibitem{Andre2001a}
{\sc {Andr\'e}, C.A.M.}
\newblock {The basic character table of the unitriangular group.}
\newblock {\em {J. Algebra} 241}, 1 (2001), 437--471.

\bibitem{Andre2002a}
{\sc {Andr\'e}, C.A.M.}
\newblock {Basic characters of the unitriangular group (for arbitrary primes).}
\newblock {\em {Proc. Am. Math. Soc.} 130}, 7 (2002), 1943--1954.

\bibitem{Andre2013a}
{\sc {Andr\'e}, C.A.M.}
\newblock Supercharacters of unitriangular groups and set partition
  combinatorics.
\newblock CIMPA school: ``Modern Methods in Combinatorics ECOS2013'' (``2da
  Escuela Puntana de Combinatoria: Escuela de Combinatoria del Sur'',
  Universidad Nacional de San Luis, Argentina, 2013.

\bibitem{Andre2018a}
{\sc Andr\'{e}, C.A.M., Gomes, F., and Lochon, J.}
\newblock Indecomposable supercharacters of the infinite unitriangular group.
\newblock In {\em Operator theory, operator algebras, and matrix theory},
  vol.~267 of {\em Oper. Theory Adv. Appl.} Birkh\"{a}user/Springer, Cham,
  2018, pp.~1--24.

\bibitem{Baggett1997a}
{\sc Baggett, L.W., Kaniuth, E., and Moran, W.}
\newblock Primitive ideal spaces, characters, and {K}irillov theory for
  discrete nilpotent groups.
\newblock {\em J. Funct. Anal. 150}, 1 (1997), 175--203.

\bibitem{Bekka2020b}
{\sc Bekka, B.}
\newblock The {P}lancherel formula for countable groups, 2020.
\newblock Preprint available at \texttt{https://arxiv.org/pdf/2009.01065v2.pdf}

\bibitem{Bekka2020c}
{\sc Bekka, B., and de~la Harpe, P.}
\newblock {\em Unitary Representations of Groups, Duals, and Characters},
  vol.~250 of {\em Mathematical Surveys and Monographs}.
\newblock American Mathematical Society, Providence, RI, 2020.

\bibitem{Bekka2008a}
{\sc Bekka, B., de~la Harpe, P., and Valette, A.}
\newblock {\em Kazhdan's property ({T})}, vol.~11 of {\em New Mathematical
  Monographs}.
\newblock Cambridge University Press, Cambridge, 2008.

\bibitem{Blattner1963a}
{\sc Blattner, R.J.}
\newblock Positive definite measures.
\newblock {\em Proc. Amer. Math. Soc. 14\/} (1963), 423--428.

\bibitem{Corwin1994a}
{\sc Corwin, L., and Pfeffer~Johnston, C.}
\newblock On factor representations of discrete rational nilpotent groups and
  the {P}lancherel formula.
\newblock {\em Pacific J. Math. 162}, 2 (1994), 261--275.

\bibitem{Diaconis2008a}
{\sc {Diaconis}, P., and {Isaacs}, I.}
\newblock {Supercharacters and superclasses for algebra groups.}
\newblock {\em {Trans. Am. Math. Soc.} 360}, 5 (2008), 2359--2392.

\bibitem{Dixmier1977a}
{\sc Dixmier, J.}
\newblock {\em {$C\sp*$}-algebras}.
\newblock North-Holland Publishing Co., Amsterdam-New York-Oxford, 1977.
\newblock Translated from the French by Francis Jellett, North-Holland
  Mathematical Library, Vol. 15.

\bibitem{Dixmier1981a}
{\sc Dixmier, J.}
\newblock {\em von {N}eumann algebras}, vol.~27 of {\em North-Holland
  Mathematical Library}.
\newblock North-Holland Publishing Co., Amsterdam-New York, 1981.
\newblock With a preface by E. C. Lance, Translated from the second French
  edition by F. Jellett.

\bibitem{Dudko2011a}
{\sc Dudko, A.}
\newblock Characters on the full group of an ergodic hyperfinite equivalence
  relation.
\newblock {\em J. Funct. Anal. 261}, 6 (2011), 1401--1414.

\bibitem{Edwards1975a}
{\sc Edwards, D.A.}
\newblock Syst\`emes projectifs d'ensembles convexes compacts.
\newblock {\em Bull. Soc. Math. France 103}, 2 (1975), 225--240.

\bibitem{Fabian2011a}
{\sc {Fabian}, M., {Habala}, P., {H\'ajek}, P., {Montesinos}, V., and {Zizler},
  V.}
\newblock {\em {Banach space theory. The basis for linear and nonlinear
  analysis.}}
\newblock Berlin: Springer, 2011.

\bibitem{Folland1995a}
{\sc Folland, G.B.}
\newblock {\em A course in abstract harmonic analysis}.
\newblock Studies in Advanced Mathematics. CRC Press, Boca Raton, FL, 1995.

\bibitem{Goodearl1986a}
{\sc Goodearl, K.R.}
\newblock {\em Partially ordered abelian groups with interpolation}, vol.~20 of
  {\em Mathematical Surveys and Monographs}.
\newblock American Mathematical Society, Providence, RI, 1986.

\bibitem{Hirai2005b}
{\sc Hirai, T., and Hirai, E.}
\newblock Characters of wreath products of finite groups with the infinite
  symmetric group.
\newblock {\em J. Math. Kyoto Univ. 45}, 3 (2005), 547--597.

\bibitem{Isaacs1995a}
{\sc {Isaacs}, I.}
\newblock {Characters of groups associated with finite algebras.}
\newblock {\em {J. Algebra} 177}, 3 (1995), 708--730.

\bibitem{Johnson1966a}
{\sc Johnson, R.A.}
\newblock On product measures and {F}ubini's theorem in locally compact space.
\newblock {\em Trans. Amer. Math. Soc. 123\/} (1966), 112--129.

\bibitem{Kirillov1962a}
{\sc Kirillov, A.A.}
\newblock Unitary representations of nilpotent {L}ie groups.
\newblock {\em Uspehi Mat. Nauk 17}, 4 (106) (1962), 57--110.

\bibitem{Lindenstrauss2001a}
{\sc Lindenstrauss, E.}
\newblock Pointwise theorems for amenable groups.
\newblock {\em Invent. Math. 146}, 2 (2001), 259--295.

\bibitem{Murray1943a}
{\sc Murray, F.J., and von Neumann, J.}
\newblock On rings of operators. {IV}.
\newblock {\em Ann. of Math. (2) 44\/} (1943), 716--808.

\bibitem{Phelps2001a}
{\sc Phelps, R.R.}
\newblock {\em Lectures on {C}hoquet's theorem}, second~ed., vol.~1757 of {\em
  Lecture Notes in Mathematics}.
\newblock Springer-Verlag, Berlin, 2001.

\bibitem{Pier1984a}
{\sc Pier, J.-P.}
\newblock {\em Amenable locally compact groups}.
\newblock Pure and Applied Mathematics (New York). John Wiley \& Sons, Inc.,
  New York, 1984.
\newblock A Wiley-Interscience Publication.

\bibitem{Reed1980a}
{\sc Reed, M., and Simon, B.}
\newblock {\em Methods of modern mathematical physics. {I}}, second~ed.
\newblock Academic Press, Inc. [Harcourt Brace Jovanovich, Publishers], New
  York, 1980.
\newblock Functional analysis.

\bibitem{Rudin1987a}
{\sc Rudin, W.}
\newblock {\em Real and complex analysis}, third~ed.
\newblock McGraw-Hill Book Co., New York, 1987.

\bibitem{Thoma1964a}
{\sc Thoma, E.}
\newblock \"{U}ber unit\"{a}re {D}arstellungen abz\"{a}hlbarer, diskreter
  {G}ruppen.
\newblock {\em Math. Ann. 153\/} (1964), 111--138.

\bibitem{Thoma1967a}
{\sc Thoma, E.}
\newblock \"{U}ber das regul\"{a}re {M}ass im dualen {R}aum diskreter
  {G}ruppen.
\newblock {\em Math. Z. 100\/} (1967), 257--271.

\bibitem{Thoma1968a}
{\sc Thoma, E.}
\newblock Eine {C}harakterisierung diskreter {G}ruppen vom {T}yp {I}.
\newblock {\em Invent. Math. 6\/} (1968), 190--196.

\bibitem{Tonti2019a}
{\sc Tonti, F.E., and T\"{o}rnquist, A.}
\newblock A short proof of {T}homa's theorem on type {I} groups, 2019.
\newblock Preprint available at \texttt{https://arxiv.org/pdf/1904.08313.pdf}

\bibitem{Varadarajan1963a}
{\sc Varadarajan, V.S.}
\newblock Groups of automorphisms of {B}orel spaces.
\newblock {\em Trans. Amer. Math. Soc. 109\/} (1963), 191--220.

\bibitem{Vershik1982a}
{\sc {Vershik}, A.M., and {Kerov}, S.}
\newblock {Asymptotic theory of characters of the symmetric group.}
\newblock {\em {Funct. Anal. Appl.} 15\/} (1982), 246--255.

\bibitem{Vershik1981a}
{\sc Vershik, A.M., and Kerov, S.V.}
\newblock Asymptotic theory of the characters of a symmetric group.
\newblock {\em Funktsional. Anal. i Prilozhen. 15}, 4 (1981), 15--27, 96.

\bibitem{Walters1982a}
{\sc Walters, P.}
\newblock {\em An introduction to ergodic theory}, vol.~79 of {\em Graduate
  Texts in Mathematics}.
\newblock Springer-Verlag, New York-Berlin, 1982.

\bibitem{Yan2001a}
{\sc {Yan}, N.}
\newblock {\em Representation theory of the finite unipotent linear groups}.
\newblock PhD thesis, University of Pennsylvania, USA, 2001.

\end{thebibliography}

\end{document}